\documentclass[12pt]{amsart}

\usepackage{amsmath}
\usepackage{amsthm}
\usepackage{amsfonts}
\usepackage{amssymb}
\usepackage[all]{xy}
\usepackage{url}
\usepackage{comment}

\newcommand\R{\mathbb{R}}

\newcommand\Z{\mathbb{Z}}

\newcommand\Id{\textup{id}}

\newcommand{\defin}[1]{\textbf{#1}}

\newcommand\acs{\varphi}  				
\newcommand\acxs{J}						
\newcommand\scxs{\zeta}					
\newcommand\bscxs{\bar \zeta}			

\newcommand\im{\textup{Im}}

\newcommand\Num{\textup{Num}}
\newcommand\Denom{\textup{Denom}}
\newcommand\del{\partial}

\newcommand{\an}[1]{\langle{#1}{\rangle}}
\newcommand{\wh}{\widehat}

\newcommand\BN{\mathcal{B}}             
\DeclareMathOperator{\ord}{ord}

\newtheorem{Theorem}{Theorem}[section]
\newtheorem{Definition}[Theorem]{Definition}
\newtheorem{Lemma}[Theorem]{Lemma}
\newtheorem{Proposition}[Theorem]{Proposition}

\newtheorem{Corollary}[Theorem]{Corollary}

\theoremstyle{remark}
\newtheorem{Example}[Theorem]{Example}
\newtheorem{Remark}[Theorem]{Remark}

\usepackage{color}
\newcommand{\jbcomm}[1]{\begingroup\color{green}JB:~#1\endgroup}
\newcommand{\dccomm}[1]{\begingroup\color{blue}DC:~#1\endgroup}
\newcommand{\ascomm}[1]{\begingroup\color{red}AS:~#1\endgroup}

\advance \topmargin by -\headheight
\advance \topmargin by -\headsep
     
\evensidemargin \oddsidemargin
\begin{document}
\title[The topology of Stein fillable manifolds in high dimensions II.]{The topology of Stein fillable manifolds in high dimensions II.}

\author[Bowden]{Jonathan Bowden}
\address{Mathematisches Institut, Universit\"{a}t Augsburg, Universit\"{a}tstr 14, 86159 Augsburg, Germany}
\email{jonathan.bowden@math.uni-augsburg.de}
\author[Crowley]{Diarmuid Crowley}
\address{Intitute of Mathematics, University of Aberdeen, Aberdeen AB24 3UE, United Kingdom}
\email{dcrowley@abdn.ac.uk}
\author[Stipsicz]{Andr\'{a}s I. Stipsicz, with an appendix by Bernd C.~Kellner}
\address{ R\'{e}nyi Institute of Mathematics, Re\'{a}ltanoda u.~13-15., Budapest, Hungary H-1053}
\email{stipsicz@renyi.hu}
\address{Mathematisches Institut, Universit\"{a}t G\"{o}ttingen, Bunsenstr.\ 3–-5, 37073 G\"{o}ttingen, Germany}
\email{bk@bernoulli.org}


\maketitle

\begin{abstract}
We continue our study of contact structures on manifolds of dimension
at least five using surgery theoretic methods.  Particular
applications include the existence of `maximal' almost contact
manifolds with respect to the Stein cobordism relation as well as the
existence of weakly fillable contact structures on the product
$M\times S^2$. 
We also study
the connection between Stein fillability and connected sums: we give
examples of almost contact manifolds for which the connected sum is
Stein fillable, while the components are not.

Concerning obstructions to Stein fillability, we show for all $k >1$
that there are almost contact structures on the $(8k{-}1)$-sphere
which are not Stein fillable.  This implies the same result for all
highly connected $(8k{-}1)$-manifolds which admit almost contact
structures.  The proofs rely on a new number theoretic result about
Bernoulli numbers.
\end{abstract}

\section{Introduction} 
\label{sec:introduction}

One of the driving questions in contact topology was to determine
which smooth closed oriented manifolds $M$ of dimension $2q{+}1$ admit a
\emph{contact structure}, where a (coorientable) contact structure is a
codimension{-}$1$ distribution $\xi$ that is defined as the kernel of a 1-form
$\lambda \in \Omega ^1 (M)$ with the property that $\lambda \wedge
(d\lambda )^q$ is a positive volume form. Since a contact structure
splits the tangent bundle of the $(2q{+}1)$-manifold $M$ as the direct sum
of a trivial real line bundle and a complex $q$-dimensional subbundle,
we need to assume that the manifold in question is already equipped
with such a splitting, called an \emph{almost contact} structure. The
general existence question for almost contact manifolds was recently answered by
Borman{-}Eliashberg-Murphy:
  
  \begin{Theorem}[\cite{BEM}]
  Suppose that $(M, \varphi )$ is a closed oriented  $(2q{+}1)$-dimensional almost 
  contact manifold. Then there is a contact structure on $M$ homotopic to the given almost contact 
  structure.\qed 
  \end{Theorem}
  
  Indeed, the construction of \cite{BEM} provides a contact structure which contains an \emph{overtwisted disk} (cf.\,\,\cite[Section~2.5]{BEM}),
  so in particular it is not symplectically fillable in any sense. (For various notions of 
  symplectic fillability, see Section~\ref{subsec:fillandsurg} and \cite{Massot12}.)
  For this reason, constructions of fillable structures, and obstructions for their existence
  seem essential in an effort to understand all contact structures (up to 
  contactomorphism or contact isotopy) on a given almost contact $(2q{+}1)$-manifold.
  Such complete classifications are available for some
  classes of 3-dimensional manifolds, although the complete picture is still to be discovered even in that dimension.
  
The strongest fillability notion is provided by \emph{Stein fillability}. 
Recall that a compact complex
manifold $W$ is a  \emph{Stein domain}
if it  admits a strictly plurisubharmonic function for which the
boundary is a regular level set. According to Eliashberg's
characterization, a $2n$-manifold with $n\geq 3$ admits a Stein
structure if and only if it admits an almost complex structure and a
handle decomposition with handles of index at most $n$ 
\cite{Cieliebak&Eliashberg12, Eliashberg??}. A Stein
structure on $W$ naturally induces a contact structure on $M=\partial
W$, and contact structures presentable in this way are called
\emph{Stein fillable}. 
Using the above topological characterization of Stein domains, 
modified surgery theory can be fruitfully applied in studying Stein fillability as in \cite{BCS2}. 

This topological characterization of Stein domains easily generalizes to 
cobordisms, providing 
the relation of
\emph{topological Stein cobordism} for almost contact manifolds: Two
almost contact manifolds $(M_0,\varphi _0)$ and $(M_1,\varphi_1)$ are in
this relation if there is an almost complex cobordism between them
that is compatible with $\varphi _0$ and 
$\varphi _1$ on the two ends, and admits a relative handle decomposition with
handles of index at most half the dimension when built on
$M_0\times [0,1]$. For convenience we write
$(M_0,\varphi _0) \prec (M_1,\varphi_1)$ in this case.  (Notice that
this relation is not symmetric.)  

A surprising application of the surgery theoretic approach to 
existence problems in contact topology provides the following result
about topological Stein cobordisms.
(For a more precise statement, see Proposition~\ref{prop:final_elements}.)

\begin{Theorem}\label{thm:maximalelement}
For a fixed dimension $2q{+}1\geq 5$ there is an almost contact
$(2q{+}1)$-manifold $(M _{max}, \varphi _{max})$ such that for any
almost contact $(2q{+}1)$-manifold $(M, \varphi )$ we have
\[
(M,\varphi ) \prec (M_{max},\varphi_{max}).
\]
\end{Theorem}
\noindent This theorem should be compared with a result of
Etnyre-Honda~\cite{EtHond}, showing that in dimension \emph{three} there are \emph{initial} contact manifolds so that $(M_{min}, \xi _{min})$ is
Stein cobordant to any other contact manifold $(M, \xi )$. In the case
of almost contact $5$-manifolds whose almost contact structures have
vanishing first Chern class, one can even take
$(M_{max},\varphi_{max}) = (S^5,\varphi_{std})$ (cf.\ Proposition
\ref{Calabi_Yau}).

The notion of topological Stein cobordism introduced above
allows for the following interpretation of the main result of \cite{BCS1}.
 Recall that according to a result of Bourgeois \cite{Bou},  
for a
contact manifold $(M, \xi )$ the product $M\times T^2$ carries an 
almost contact structure $\varphi _{T}$ which can be represented by a contact
structure $\xi _{T}$. By \cite[Example~5]{Massot12}, if $(M, \xi )$ is weakly fillable
then so is $\xi _{T}$. Since the main result of 
\cite{BCS1} shows that for some appropriately chosen almost contact structure
$\varphi _{S}$ on $M\times S^2$ we have 
$(M\times T^2,  \varphi _T) \prec (M\times S^2, \varphi _S)$, this yields
the following
variant of the main result of \cite{BCS1}:

\begin{Theorem}\label{thm:mandmtimess2}
Suppose that a contact manifold $(M, \xi )$ 
admits a weak symplectic filling $(W, \omega )$.
Then the product $M\times S^2$ admits a
weakly fillable contact structure.
\end{Theorem}

\begin{Remark}
Note that while all our other statements are concerned with 
\emph{almost} contact structures on manifolds, Theorem~\ref{thm:mandmtimess2}
is about genuine contact structures.
\end{Remark}

We now move from products to connected sums.
The connected sum of two 3-manifolds is Stein fillable if and only if
both 3-manifolds are Stein fillable \cite{Eliashberg90}. 
Recall that in higher dimensions, the diffeomorphism types of components of a connected 
sum are only
well-defined up to connect summing with homotopy spheres. In contrast
to dimension three, we have the following result.

\begin{Theorem}\label{thm:mainconnsum}
Let $M = ST^*S^{2k{+}1}$ be the unit cotangent bundle of the
$(2k{+}1)$-sphere.  For every odd $k \geq 5$, $M$ admits an almost
contact structure $\varphi$ such that $(M,\varphi) \# (-M,-\varphi)$
admits a Stein fillable contact structure.  However, for every
almost contact homotopy $(4k{+}1)$-sphere $(\Sigma, \acs_\Sigma)$,
neither $(M \# \Sigma, \varphi \# \varphi_{\Sigma})$ nor 
$\bigl(-(M \# \Sigma), -(\varphi \# \varphi_{\Sigma}) \bigr)$ is Stein fillable.
%
%
%
\end{Theorem}
%


Dimension five appears to be 
intermediate between dimension three and higher dimensions,
with regard to the Stein fillability of the summands of a Stein fillable
connected sum.
In dimension five, there are no exotic spheres, and 
if $(M, \acs)$ is an almost contact manifold where $M$ is a connected sum
$M = M_0 \# M_1$, then $(M, \acs) = (M_0, \acs) \# (M_1, \acs_1)$ for almost
contact structures $\acs_i$ on $M_i$ which are uniquely defined up to homotopy
(see Lemma \ref{lem:5d-almost-contact-disconnected-sum}).
By abuse of notation we let
$c_1(\varphi) \colon \pi_2(M) \to \Z = \pi_2(BU)$ denote the evaluation homomorphism given by the first Chern class of the almost contact structure $\varphi$. We then have the following
analogue of Eliashberg's theorem.

\begin{Theorem}\label{thm:5fill}
Let $(M, \acs) = (M_0 \# M_1, \acs_0 \# \acs_1)$ be a Stein fillable almost contact $5$-manifold.
Assume that either $c_1(\acs) = 0$ or that
\[ c_1(\acs)(\pi_2(M_0)) = c_1(\acs)(\pi_2(M_1)) = \Z = \pi_2(BU).\]
Then both $(M_0, \acs_0)$ and $(M_1, \acs_1)$ are Stein fillable.
%
%
%
\end{Theorem}
We next consider Stein fillability
of almost contact structures on  spheres. Let $(S^{2q{+}1}, \zeta_{std})$ be
the {\em standard} stable almost contact structure on the $(2q{+}1)$-dimensional sphere, which is induced by the Stein $(2q{+}2)$-disk.
When $2q{+}1 = 8k{-}1$, basic
obstruction theory shows that 
$S^{8k{-}1}$ has two stable almost contact structures, $\zeta _{std}$ and 
a non{-}standard or \emph{exotic} stable almost contact structure $\zeta _{ex}$.
The exotic structure $\zeta _{ex}$ is
harder to visualize than $\zeta_{std}$ (see Section~\ref{subsec:explicit} for
a description when $k > 1$).  It follows from \cite{BCS2, Geiges93} that 
$(S^7, \zeta_{ex})$ can be represented by a Stein fillable contact structure. 
In contrast, for $8k{-}1>7$ we have 

\begin{Theorem} \label{thm:wackyspheres}
The exotic stable almost contact structure $\zeta _{ex}$ on 
$S^{8k{-}1}$ cannot be represented by a Stein fillable contact structure once $k\geq 2$.
\end{Theorem}


Theorem \ref{thm:wackyspheres} rests on Theorem \ref{thm:YangIntro} below, which improves a result of Yang \cite{Yang12} 
about the existence of stable almost complex structures on $(4k{-}1)$-connected $8k$-manifolds.
Before stating these results we first recall some notation and terminology. Let $F \colon BU \to BSO$ be the forgetful map between the classifying spaces for stable unitary 
and stable oriented vector bundles.
A  necessary condition
for an oriented manifold $X$
to admit a stable complex structure
is that $$\im(\tau_{X*}) \subseteq F_*(\pi_{4k}(BU)) \subseteq \pi_{4k}(BSO),$$
where $\tau_{X*} \colon \pi_{4k}(X) \to \pi_{4k}(BSO)$ is induced by 
the classifying map of the stable tangent bundle of $X$, $\tau_X \colon X \to BSO$.
(Note that when $k$ is even, $\pi_{4k}(BSO)/F_*(\pi_{4k}(BU)) = \Z/2$.)
According to the following theorem, once $k>1$, this necessary condition is also
sufficient. 

\begin{Theorem} \label{thm:YangIntro}
A smooth closed oriented $(4k{-}1)$-connected $8k$-manifold\,\,\,$Y$\! admits
a stable almost complex structure if and only if
\begin{enumerate}
\item $k\geq 3$ is odd, or
\item  $k=1$ and the  signature $\sigma_Y$ of $Y$  is even, or
\item $k$ is even and  $\im(\tau_{Y*}) \subseteq F_*(\pi_{4k}(BU))$.
\end{enumerate}
\end{Theorem}
The improvement provided by Theorem~\ref{thm:YangIntro} over
Yang's result 
is the removal of assumptions involving Bernoulli numbers.
This step is made possible by a new divisibility property
of differences of reciprocals of Bernoulli numbers, which is proven in the Appendix
written by Bernd Kellner.


Computing the appropriate bordism obstruction class to Stein fillability from \cite{BCS2},
Theorem~\ref{thm:wackyspheres} implies the following non{-}fillability result for highly
connected manifolds:


\begin{Corollary}\label{cor:wacky}
Let $M$ be a $(4k{-}2)$-connected $(8k{-}1)$-manifold and $k \geq 2$.
Suppose that $M$ admits an almost contact structure.
Then $M$ admits an almost contact structure which cannot be
represented by any Stein fillable contact structure.
\end{Corollary}

{\bf {Outline of the paper:}} In Section~\ref{sec:fillability}
we review some basic notions and recall the definition of the
obstruction class as introduced in \cite{BCS2} associated to an almost contact
manifold. We also prove
Theorems~\ref{thm:mandmtimess2} in this section.  In
Section~\ref{sec:max} we present the proof
of Theorem~\ref{thm:maximalelement}.
Section~\ref{subsec:stein_fillability_and_connected_sums} is devoted
to the study of the relation between Stein fillability and connected
sums, and in particular it contains the proofs of
Theorems~\ref{thm:mainconnsum} and~\ref{thm:5fill}.  In
Section~\ref{sec:nonst-acs} we examine the Stein fillability of stable almost
contact structures on $(8k{-}1)$-spheres, and in particular prove
Theorem~\ref{thm:wackyspheres} and Corollary~\ref{cor:wacky}. 
Section~\ref{sec:nonst-acs} also contains our improvement 
of Yang's result given in Theorem~\ref{thm:YangIntro}
about existence of stable almost complex structures on
highly connected $8k$-manifolds.
The
Appendix contains the number theoretic
result about Bernoulli numbers needed for the proof of
Theorem~\ref{thm:wackyspheres}, and was written by Bernd Kellner.

\bigskip

{\bf {Acknowledgements:}} The authors would like to thank the
Max-Planck-Institute in Bonn and the Laboratoire de Math\'{e}matiques Jean Leray in Nantes
for their hospitality which enabled parts of this work to be carried out.  
We would also like to thank 
Pieter Moree for providing a bridge to the world of number theory
and contacting Karl Dilcher and Bernd Kellner. 
We are grateful to the referee for many helpful comments and suggestions.
JB was partially supported by DFG Grant BO4423/1-1.
DC acknowledges the support  of the Leibniz 
Prize of Wolfgang L\"{u}ck, granted by the Deutsche Forschungsgemeinschaft.
AS was partially supported by OTKA
K100796, by the \emph{Lend\"ulet program} of the Hungarian Academy of
Sciences and by ERC Advanced Grant LDTBud.    
The present work is part of the authors' activities within CAST, a Research Network Program of the European
Science Foundation.


\section{Fillability and surgery} 
\label{sec:fillability}
In their proof of the existence of contact structures on closed almost contact manifolds
\cite{BEM}, Borman{-}Eliashberg-Murphy produce contact structures with the
additional property that they are not fillable in any sense. For this reason we will focus
on finding fillable structures on various manifolds.

\subsection{Fillable structures and Stein cobordisms}
\label{subsec:fillandsurg}
We begin by recalling the definitions of the various standard notions
of fillability of contact structures. For a more  detailed account we
refer to \cite{Geiges08, Massot12}. Recall that a symplectic manifold $(W,
\omega)$ is a $(2q{+}2)$-dimensional manifold $W$ with a closed $2$-form
$\omega$ such that $\omega^{q{+}1} \neq 0$. Hence a symplectic manifold
carries a canonical orientation. Similarly, a cooriented contact
structure $\xi$ on a $(2q{+}1)$-manifold $M$ determines an orientation
of $M$ given by the form $\lambda \wedge (d \lambda)^q$.

\begin{Definition} \label{def:contact}
A contact manifold $(M,\xi)$ is \textbf{weakly symplectically
  fillable} if it is the oriented boundary of a compact symplectic manifold
$(W,\omega)$ and there is an almost complex structure $J$ that is
tamed by $\omega$ so that 
 $J(TM) \cap TM = \xi$ and for a contact form
$\lambda$ defining $\xi $ we have 
$ d \lambda (v,Jv) > 0$ (for all  $0\neq v \in \xi$).
\end{Definition}

\noindent This definition was introduced in \cite{Massot12}, where it
was shown to be strictly weaker than the more standard notion of
strong fillability.
\begin{Definition}
A contact manifold $(M,\xi)$ is called \textbf{strongly symplectically
  fillable} if it bounds a compact symplectic manifold $(W,\omega)$ and there
is an outward pointing vector field $V$ near $\del W$ such that the
Lie derivative satisfies $L_V\omega = \omega$ and $\lambda
=\iota_V\omega$ is a defining $1$-form for $\xi$.  If the symplectic
form $\omega$ is also exact then we say that $(M,\xi)$ is \textbf{exactly
  fillable}.
\end{Definition}
\noindent Note that strong fillability is equivalent to weak
fillability plus the condition that the symplectic form is exact near
the boundary \cite[Remark 1.11]{Massot12}. A further specialisation of
the fillability notion is that of Stein fillability. Recall that a
\emph{Stein domain} is a compact, complex manifold $(W,J)$ with
boundary that admits a function $f\colon W \to [0,1]$ so that
$f^{-1}(1) = \del W$ is a regular level set and $\omega = -dd^{\mathbb
  {C}} f$ is a symplectic form (where $d^{\mathbb {C}} f(X) = d
f(JX)$).
\begin{Definition}
A contact manifold $(M,\xi)$ is \textbf{Stein fillable} if it bounds a
Stein domain $(W,J)$ such that $\xi = J(TM) \cap TM$.
\end{Definition}
\noindent These notions of fillability fit into the 
following sequence of inclusions of contactomorphism classes of contact manifolds, all of which are known to be strict:
\begin{equation*} \label{eq:contactflavours}
\{\text{Stein fillable\} $\subset$ \{exactly fillable\} $\subset$
  \{strongly fillable\} $\subset$ \{weakly fillable\}}.
\end{equation*}

The applicability of surgery theoretic methods in the study of
fillable contact structures is provided by the following fundamental
result of Eliashberg:

\begin{Theorem}[Eliashberg's $h$-principle, \cite{Cieliebak&Eliashberg12, Eliashberg??}]
\label{thm:h-principle}
	Let $(W,J)$ be a compact $(2q{+}2)$-di\-men\-si\-o\-nal almost complex
        manifold admitting a handle decomposition with handles of index $q{+}1$ or less, and
        suppose that $q\geq 2$.  Then $J$ is homotopic to a complex
        structure $\widetilde{J}$ so that $(W,\widetilde{J})$ is a Stein
        filling of a contact structure $\xi$ on $M = \partial W$. \qed
\end{Theorem}

The concept of Stein domains can be generalized to 
cobordisms as follows:

\begin{Definition}
A smooth cobordism $W$ between contact 
manifolds $(M_0, \xi _0)$ and $(M_1, \xi _1)$ is a \textbf{Stein cobordism}
if 
\begin{itemize}
\item $\partial W=-M_0 \sqcup M_1$;
\item $W$ admits a complex structure $J$ and a map $f\colon W\to [0,1]$ such that 
$M_0 :=f^{-1}(0)$ and $M_1:=f^{-1}(1)$ are regular level sets;
\item $\omega =-dd^{\mathbb {C}}f$ is a symplectic form;
\item $\xi_i = J(TM_i)\cap TM_i$, i.e.\ the complex structure $J$ 
induces the contact structures $\xi _i$ on the ends of the cobordism. 
The contact
manifold $(M_0,\xi _0)$ is usually called the \emph{concave} end and
$(M_1, \xi _1)$ the \emph{convex} end of the Stein cobordism $(W, J)$.
\end{itemize}
\end{Definition}

The proof of  
Theorem~\ref{thm:h-principle} proceeds by inductively adding handles to the standard
contact structure on the sphere $S^{2q{+}1}$ (which is regarded as the 
boundary of the standard complex ball), and showing that the traces of these
handle attachments can be endowed with the structure of a Stein cobordism:

\begin{Theorem}\label{thm:h-principleforhandles}
Let $(M^{2q{+}1},\xi)$ be a contact manifold of dimension $2q{+}1 \geq 5$.
Suppose that $k \leq q{+}1$ and that $M'$ is obtained from $M$ 
via an 
almost complex handle attachment  of index $k$ with trace $(M \times I) \cup h_{k}$. 
Then the almost complex structure $J$ on the trace 
is homotopic to a complex structure $\widetilde{J}$ so that
$((M\times I)\cup h_{k}, \widetilde{J})$ 
is a Stein cobordism  from $(M, \xi)$ to $(M',\xi')$ (with some contact structure $\xi '$ on 
$M'$). \qed
\end{Theorem}

\begin{Remark}
By equipping the product $M_0\times I$ with the symplectic structure given by the symplectization $\omega _{sp}(\xi _0)$ 
of the contact 
structure $\xi _0$ and isotoping the attaching sphere of the handle $h_k$ to an
isotropic sphere, the symplectic form $\omega _{sp}(\xi _0)$ was extended by 
Weinstein~\cite{Weinstein} to the trace $(M_0\times I )\cup h_k$. The existence of a Stein 
structure on the trace (in particular, the construction of the appropriate function $f$ of
the definition) is due to Eliashberg~\cite{Cieliebak&Eliashberg12, Eliashberg??}. 
When the symplectic or Stein structures are implicitly assumed in our later arguments, we will refer to such 
handles and handle attachments as Stein/Weinstein handles resp.\ handle attachments.
\end{Remark} 

In the following we would like to emphasize the topological nature of the above definitions.
To do this in the proper setting, we need to recall the definitions of almost contact and
stably complex structures and manifolds.

Suppose that $M$ is a smooth closed oriented
$(2q{+}1)$-manifold and $\varphi$ is an almost contact structure on
$M$. The tangent bundle of $M$ is classified by the the map $\tau
\colon M\to BSO(2q{+}1)$, and an almost contact structure provides a
lift of this map to $BU(q)$:
\[ \xymatrix{  & BU(q) \ar[d]^-{F_q}\\
M \ar[r]^-{\tau} \ar[ur]^-{\varphi} & BSO (2q{+}1),} \]
where $F_q$ is induced by the canonical embedding $U(q)\to SO(2q{+}1)$. All these maps can be stabilized to yield maps to $BSO$ resp.\  $BU$. 
For some purposes, it is helpful to formulate results using the \emph{stable normal Gauss map}
$\nu\colon M\to BSO$ rather than the tangential map $\tau$, and we
will follow this strategy.
In this setting, a map $\zeta \colon M \to BU$ in a commutative diagram
\[ \xymatrix{  & BU \ar[d]^-F\\
M \ar[r]^-{\nu} \ar[ur]^-{\zeta} & BSO } 
\] 
describes a complex structure on the normal bundle of $M$.
Since the sum of the stable tangent and normal bundles is canonically trivialized, 
a normal complex structure determines a unique {\em stable complex} (or {\em stable contact}) structure,
and \emph{vica versa}.
Theorem~\ref{thm:h-principleforhandles} motivates the following definition:

\begin{Definition} \label{def:topological-Stein{-}cobordism}
A stably almost contact $(2q{+}1)$-manifold  $(M_0,\zeta_0)$ is \defin{topologically 
Stein cobordant} to $(M_1,\zeta_1)$ if there is a stably complex  cobordism $(W,\zeta)$ 
such that 
$$\partial (W, \zeta) = -(M_0,\zeta_0) \sqcup (M_1,\zeta_1)$$
as stably complex manifolds and $W$ is built from $M_0\times [0,1]$ 
by attaching handles of index $\le q{+}1$. In this case we write
$$(M_0,\zeta_0) \prec (M_1,\zeta_1),$$
and call $(W, \zeta )$ a \defin{topological Stein cobordism}.
\end{Definition}
Note that according to \cite[Lemma 3.6]{BCS2} the Stein cobordism
relation is the same if we consider almost complex cobordisms or
stably complex ones. This follows from the fact that every almost
complex structure in a given stable class can be realized by taking
the connected sum with various Stein fillable almost contact structures on
the standard sphere. Hence we can also consider the Stein cobordism
relation given by a true almost complex bordism $(W,J)$ between almost
contact manifolds $(M_0,\varphi_0) $ and $ (M_1,\varphi_1)$. In short,
if $\varphi _i$ generates the stable complex structure $\zeta _i$, then
\begin{equation} \label{eq:stable_and_unstable_Stein_cobordism}
(M_0,\varphi_0) \prec (M_1,\varphi_1) \Longleftrightarrow
(M_0,\zeta_0) \prec (M_1,\zeta_1).
\end{equation}
If, in addition, an almost contact structure $\varphi _0$ on
$M_0$ is represented by a contact structure, then repeated
application of Theorem~\ref{thm:h-principleforhandles} shows that $J$ can be homotoped to a Stein
structure on the cobordism $W$.

A Stein cobordism from $(M_0, \xi _0)$ to $(M_1, \xi _1)$ can be glued
to a Stein filling of $(M_0, \xi _0)$, providing a Stein filling of
$(M_1, \xi _1)$. Attaching Stein/Weinstein handles preserves strong
fillability, hence gluing a Stein cobordism to a strong filling again
yields a strong filling. 
When gluing a Stein cobordism to a weak symplectic filling, however,
some care is needed: as shown by the next lemma, we need to assume
that the symplectic form vanishes on the attaching spheres of
$3$-handles.

\begin{Lemma}\label{weak_fill}
Let $(W ,\omega)$ be a weak filling of a contact manifold $(M ,\xi_0)$
and suppose that $(W_1, J)$ is a Stein cobordism from $(M, \xi _0)$ to
$(M_1, \xi _1)$ consisting of a single $k$-handle attachment so that
$\omega$ vanishes on the homology class of the
attaching sphere if $k =3$. Then $W' = W\cup
W_1$ (equipped with a suitable symplectic structure $\omega'$, based
on $\omega $ and the Stein structure on $W_1$) provides a weak filling
of $(M_1, \xi _1)$.

Furthermore, if the attaching sphere of a $2$-handle bounds a surface
$\Sigma$ in $M$ then we can assume that the $\omega'([\Sigma \cup
  D^2]) = 0$,\ where $D^2$ denotes the core of the $2$-handle.
\end{Lemma}

\begin{proof}
Let $(-\epsilon,\epsilon) \times M$ be a small regular neighbourhood
of $M$ in $W$, where $W$ has been extended slightly. Let $\lambda$ be
a defining 1-form for $\xi$. Suppose that $\omega|_M$ is exact on the
attaching sphere $S^{k-1}$ of the $k$-handle of $W_1$. Then there is a form
$\overline{\omega}$ cohomologous to $\omega|_M$ which vanishes near
$S^{k-1}$. By \cite[Lemma~1.10]{Massot12} one can alter the symplectic
structure after attaching a sufficiently long end $ [0,2C] \times
\partial W$ so that the symplectic form is given by $\overline{\omega}
+ d(t \lambda)$ for all $t \ge C - \epsilon$ and we still have a weak
filling of $M = M \times \{C\}$. In particular, near $S^{k-1} \subset
M \times \{C\}$ the symplectic form is just $d(t \lambda)$. We attach
a Stein/Weinstein $k$-handle along $S^{k-1}$ and denote the resulting
filling by $(W',\omega')$. The almost complex structure $J$ on $W$
used in the definition of weak filling then extends to an almost
complex structure $J'$ on $W'$ that is tamed by $\omega'$.

Since $\omega $ is always exact near an attaching sphere $S^{k-1}$
with $k\ne 3$ and this is the case by assumption if $k=3$, the lemma
follows immediately.

In the case of a $2$-handle whose attaching sphere $S^1$ bounds a
surface $\Sigma$, we can assume that the form $\overline{\omega}$
above vanishes on $\Sigma \subset M \times \{C\}$. Then since the core
of a Weinstein handle is isotropic with isotropic boundary, it
follows that
$$\omega'([\Sigma \cup D^2]) = \int_\Sigma C \thinspace d\lambda +
\int_{D^2} \omega' = \int_{\partial \Sigma} C\thinspace \lambda = 0,$$
giving the final claim. 
\end{proof}

In \cite{BCS1} a contact structure was constructed on $M\times S^2$ by 
constructing a topological Stein cobordism between $M\times T^2$ and
$M\times S^2$ for appropriate choices of almost
contact structures.  With the above lemma at hand, this point of view
then provides the following fillability result, which corresponds to 
Theorem~\ref{thm:mandmtimess2} in the Introduction.

\begin{Proposition}\label{cor:weak_product}
Let $(M,\xi)$ be a contact manifold of dimension $2q{+}1$ 
that admits a weak symplectic filling $(W,\omega)$. Then $M \times S^2$
admits a weakly fillable contact structure.
\end{Proposition}
\begin{proof}
Let $(W,\omega)$ be a weak filling of $(M,\xi)$. By \cite[Example~5]{Massot12}
the manifold $M \times T^2$ admits a contact structure that is weakly
filled by the symplectic manifold $(W \times T^2,\omega \oplus \omega_{T^2})$. According to \cite[Proposition~3.1]{BCS1} there is a Stein cobordism $Y$ from $M \times T^2$ to $M \times S^2$ which fits into the following diagram:
\begin{equation}\label{eq:trianglewithy}
\xymatrix{ M\times T^2 \ar[dr]_(0.45){f_0} \ar[r]^(0.6){i_{0}} & Y 
\ar[d]^(0.4){g_Y} & M\times S^2 \ar[l]_(0.6){i_{1}} \ar[dl]^(0.45){Id} \\ & M\times S^2,} 
\end{equation}
where $g_Y$ is a $(q{+}2)$-equivalence and $f_0$ is the product of the identity with a map of degree $1$. 

The idea of the proof is to inductively apply Lemma \ref{weak_fill} to Stein handle
attachments which make up the bordism $Y$, starting from from $M \times T^2$.
We first find a topological Stein structure on $Y$ where we can apply 
Lemma \ref{weak_fill} to each handle attachement.
For this, we need to keep track of the cohmology class of the symplectic form of the filling,
when restricted to the outgoing boundary.  
Hence we note the equality of cohomology classes 
$[\omega|_M \oplus \omega_{T^2}]= f_0^*\bigl( [\omega|_M \oplus \omega_{S^2}] \bigl)$ 
for a symplectic form $\omega_{S^2}$ on $S^2$.

Let $\alpha,\beta \subset T^2$ be the standard generators of $\pi_1(T^2)$ which we consider in different $T^2$-fibers of $M \times T^2$. These are then null-homotopic in $Y$ and hence extend to maps of discs, which can be taken to be proper embeddings in $Y$ since the dimension of $Y$ is at least $6$. Let $Y_{2}$ be the bordism obtained by attaching a pair of $2$-handles along $\alpha$ and $\beta$. We then obtain decomposition of $Y = Y_2 \cup_{X_2} Y_3$, where $X_2$ is the upper boundary component of $Y_2$. Attaching these 2-handles yields $\pi_1(Y_2) = \pi_1(X_2) = \pi_1(M)$ by construction and also that the map $Y_2 \to M \times S^2$ is a surjection on $\pi_2$. This latter claim in obvious for classes in the $\pi_2(M)$-factor and for the class coming form the $S^2$-factor observe that the result of the surgery on the class $[\text{pt} \times T^2]$ is a spherical class that is mapped to $[\text{pt} \times S^2]$ under $(g_Y)_*$.

Since $Y_2$ is formed by $2$-handle attachments, 
it has the homotopy type of a space obtained by attaching $(2q{+}2)$-cells to $X_2$.
Hence the inclusion $X_2 \to Y_2$ is $(2q{+}1)$-connected and so 
$g_{X_2} \colon X_2 \to M \times S^2$ is $2$-connected:
Here, and for the rest of the proof, we set $g_Z := g_Y|_Z$ for any subspace $Z \subset Y$.
Since $Y_3$ is obtained from the $(2q{+}2)$-dimensional manifold $Y$ by deleting neighbourhoods of $2$-handles,
the inclusion $Y_3 \to Y$ is $(2q{+}1)$-connected and so 
$g_{Y_3} \colon Y_3 \to M \times S^2$ is at least $3$-connected.
It follows that the pair $(Y_3,X_2)$ is algebraically $2$-connected,
and we will use this later in the proof.

The fact that $Y_3$ is obtained from $Y$ by deleting 
neighbourhoods of $2$-handles has another important consequence.
Combined with the fact that $(Y, M \times S^2)$ is algebraically $(q{+}1)$-connected,
it implies that the pair $(Y_3,M \times S^2)$ is algebraically $(q{+}1)$-connected. 
Thus there is a handle decomposition of $Y_3$ relative to $M \times S^2$ containing only handles of index at least $q{+}2$ by a result of Wall \cite[Theorem 2.18]{BCS2} or dually there is a handle decomposition of $Y_3$ relative to $X_2$ containing only handles of index at most $q{+}2$. More precisely Wall shows inductively that given \emph{any} handle decomposition one can cancel handles of index $k < q{+}2$ with $(k{+}1)$-handles at the expense of introducing a $(k{+}2)$-handles. 

We now apply Lemma \ref{weak_fill} to $Y_2$ to obtain a weak filling $(Y_2,\omega')$ of $X_2$. Since the $2$-handles are attached along curves that are non-trivial in rational homology it follows from the long exact sequence of the pair $(Y_2, M \times T^2)$
that the map $M \times T^2 \to Y_2$ induces an injection on cohomology. In particular, we have the following equality for cohomology classes 
$[\omega'] = (g_{Y_2})^*\bigl( [\omega|_M \oplus \omega_{S^2}] \bigr)$. 

We now consider any topological Stein handle decomposition of $Y_3$ relative to $X_2$. As this pair is $2$-connected we can then apply Wall's argument to cancel all handles of index $k \le 1$ at the expense of introducing $3$-handles. In particular, the resulting handle decomposition will still be topologically Stein as $\dim(Y_3) = 2q{+}4 \ge 6$.

We now attach the $2$-handles $h^2_i$ of $Y_3$ to $X_2$. 
Let $Y_{2+} = Y_2 \cup \cup_{i=1}^n h^2_i \subset Y_3$ be the union of $Y_2$
and these new $2$-handles.  Let $D^2_i$ be the core of $h^2_i$. Since the pair $(Y_3, X_2)$ is $2$-connected 
the disc $D^2_i$ can be homotoped  inside $Y_3$ to a disc $\Delta_i \subset X_2$ relative to its boundary, which by general position is embedded. Applying Lemma \ref{weak_fill} for each $2$-handle we obtain a weak filling $(Y_{2+},\omega')$ such that $\omega'([D^2_i \cup \Delta_i]) = 0$.  
We claim that at the level of cohomology classes we again have 
$[\omega'] = (g_{Y_{2+}})^*\bigl( [\omega|_M \oplus \omega_{S^2}] \bigr)$.
To see this, consider the cohomology exact sequence
of the pair $(Y_{2+}, Y_2)$:
\[ \cdots \longrightarrow H^2(Y_{2+}, Y_2) \stackrel{j}\longrightarrow H^2(Y_{2+}) \longrightarrow H^2(Y_2) \longrightarrow \cdots~. \]
By construction, $[\omega']$ and $(g_{Y_{2+}})^*\bigl( [\omega|_M \oplus \omega_{S^2}] \bigr)$ 
agree when restricted to $Y_2$, so their difference lies in $j(H^2(Y_{2+}, Y_2))$,
where $H^2(Y_{2+}, Y_2)$ is a free abelian group with dual basis consisting of the $2$-handles $h^2_i$.
Since $D_i^2$ and $\Delta_i$ are homotopic relative to their boundary in $Y_3$, the spherical class $[D_i^2 \cup \Delta_i]$ is null-homotopic in $Y_3$.
This, combined with the fact that $\omega'([D^2_i \cup \Delta_i]) = 0$, ensures that
$[\omega'] - (g_{Y_{2+}})^*\bigl( [\omega|_M \oplus \omega_{S^2}] \bigr) = 0$.

We next attach $3$-handles to $Y_{2+}$. 
In order to apply Lemma \ref{weak_fill} we must ensure that the resulting symplectic form vanishes on the attaching $2$-sphere $S^2_a$. Let $\iota_{2+} \colon Y_{2+} \to Y$ be the inclusion.  
Since $S^2_a$ bounds a $3$-disc in $Y$, it follows that 
$$(g_{Y_{2+}})_*([S^2_a]) = (g_Y)_*(\iota_{2+})_*([S^2_a]) = 0 \in H_2(M \times S^2).$$
We then have
$$\langle (g_{Y_{2+}})^*\bigl( [\omega|_M \oplus \omega_{S^2}] \bigr), [S^2_a] \rangle =
\langle [\omega|_M \oplus \omega_{S^2}], (g_{Y_{2+}})_*([S^2_a]) \rangle 
= 0,$$
where the angular brackets denote the natural Kronecker pairing. Thus we can again apply Lemma \ref{weak_fill}. As above the cohomology class of the symplectic structure $\omega'$ is  just the restriction of $(g_{Y})^*\bigl( [\omega|_M \oplus \omega_{S^2}] \bigr)$. Inductively applying Lemma \ref{weak_fill} to the remaining handles completes the argument.\end{proof}




\subsection{The surgery obstruction and topological Stein cobordisms} 
\label{subsec:obstruction}
In this subsection we briefly recall the main construction
of~\cite{BCS2}. We then extend this point of view and identify the
``topological Stein envelope'' of an almost contact manifold,
i.e.\,\,those almost contact manifolds which can be obtained from a given one
via a topological Stein cobordism.

Recall that an almost contact structure $\acs$ on a $(2q{+}1)$-manifold $M$ can be regarded
as a map $\acs\colon M^{2q{+}1}\to BU(q)$, which lifts the classifying map 
$\tau \colon M^{2q{+}1}\to BSO(2q{+}1)$ of the tangent bundle of $M$. 
We then stabilize $\acs$ and pass to the corresponding complex normal
structure, which is an equivalence class of maps $\zeta \colon M\to BU$, 
which lift the stable normal Gauss map $\nu \colon M\to BSO$.
 
For a fixed integer $k$, the map $\zeta \colon M\to BU$ admits a
\emph{Postnikov factorization} $(B^k_{\zeta}, \eta ^k_{\zeta}, \bar
\zeta)$ with the following properties: these maps and spaces fit into
the commutative diagram
\[ \xymatrix{  & B^k_{\zeta} \ar[d]^{\eta^k_{\zeta}}\\
M \ar[r]^-{\zeta} \ar[ur]^-{\bar \zeta} & BU,} \]
and satisfy the following conditions:
\begin{enumerate}
\item $\eta^k_{\zeta}$ is a Serre fibration,
\item $\bar \zeta$ is a $(k{+}1)$-equivalence, that is, it induces an
isomorphism on $\pi _i$ for all $i<k{+}1$ and a surjection for $i=k{+}1$,
and
\item $\eta ^k _{\zeta}$ is a $(k{+}1)$-coequivalence, that is, it induces an
isomorphism on $\pi _i $ for $i>k{+}1$ and an injection for $i=k{+}1$.
\end{enumerate}
The existence of these spaces and maps are proved in \cite[Chapters 2 \& 5]{Baues77}.
The pair $(B^k_{\zeta}, \eta ^k _{\zeta})$ is unique (up to fiber
homotopy equivalence),  and we call them the {\em complex normal $k$-type} of the stable
complex manifold $(M,\zeta)$.
The map $\bscxs \colon M \to B^k_\scxs$ is called a {\em $\scxs$-compatible}
normal smoothing and is not, in general, uniquely determined by $\scxs$.
The only explicit complex normal $k$-types we will use in this paper 
are covered by the following

\begin{Example}[cf.~{\cite[Example 2.5]{BCS2}}] \label{exa:normal-type}
We take the stable complex bundle 
$(B^k_\scxs, \eta^k_\scxs) = (BU\an{k{+}1}, \pi_{k{+}1})$ where
the map $\pi_{k{+}1} \colon BU\an{k{+}1} \to BU$ is the $k$-fold
connective covering of $BU$. Recall that $BU\an{k{+}1}$ is the space
whose homotopy groups are trivial in degree $i \le k-1$ and such that
$\pi_{k{+}1}$ induces a surjection on the $k$-th homotopy group and
isomorphisms for all higher homotopy groups.  We denote the bordism
groups $\Omega_*(BU\an{k{+}1}; \pi_{k{+}1})$ by $\Omega_*^{U\an{k}}$. When
$k = 3$, we have  that $\Omega_*(BU\an{4}; \pi_4) = \Omega_*^{SU}$ is just
special unitary bordism as consider in \cite[Chapter X]{Stong1}.
\end{Example}

For an almost contact $(2q{+}1)$-manifold $(M,\varphi )$ with its induced
stable complex structure $\zeta$ we consider the associated complex normal
$(q{-}1)$-type $(B^{q{-}1}_{\zeta}, \eta ^{q{-}1}_{\zeta})$.
The map $\bar \zeta$ then provides a bordism class $[M, \bar \zeta]$
in the bordism group $\Omega _{2q{+}1}(B^{q{-}1}_{\zeta}; \eta
^{q{-}1}_{\zeta})$. For a detailed discussion of this group
see~\cite{BCS2}. 


{\em A priori} the bordism class $[M, \bar \zeta ]$ depends on the
choice of $(q{-}1)$-smoothing $\bscxs$, but we call any such class 
an \emph{obstruction class},
since --- according to the next theorem --- $[M, \bscxs]$ vanishes 
if and only if the almost contact structure $\varphi$ can be
represented by a Stein fillable contact structure.

\begin{Theorem}[{\cite[Theorem~1.2]{BCS2}}] 
\label{thm:stein}
A closed almost contact manifold $(M, \acs)$ of dimension
$2q{+}1\geq5$ admits a Stein fillable contact structure homotopic to
the almost contact structure $\varphi$ if and only if $[M, \bscxs] = 0
\in \Omega_{2q{+}1}(B^{q{-}1}_\scxs; \eta^{q{-}1}_\scxs)$ for any,
equivalently for all, choices of $\bscxs$, where $\zeta $ is the stabilization
of $\varphi$. \qed
\end{Theorem}

\begin{Remark} \label{rem:applying_the_obstruction_class}
The applicability of the obstruction class described above hinges on computations of
the bordism group $\Omega_{2q{+}1}(B^{q{-}1}_{\zeta_1};
\eta^{q{-}1}_{\zeta_1})$, which is a highly nontrivial matter in
general. For simply connected 7-manifolds with torsion free second homotopy group
\cite[Theorem 1.3]{BCS2} shows that $\Omega_7(B^2_{\zeta},
\eta^2_{\zeta}) = 0$; implying that all such almost contact
7-manifolds are Stein fillable.  For $(q{-}1)$-connected
$(2q{+}1)$-manifolds further calculations of these bordism groups will be presented in
\cite{BCS3}.
\end{Remark}

In terms of the topological Stein cobordism relation given in Definition 
\ref{def:topological-Stein{-}cobordism},
Theorem \ref{thm:stein} states that $(S^{2k{+}1}, \scxs_{std}) \prec (M, \scxs)$ if and only if an obstruction class $[M, \bscxs]$ vanishes.  
We now extend Theorem \ref{thm:stein} to give a bordism theoretic determination of the topological Stein cobordism relation for any pair of closed $(2q{+}1)$-dimensional stably complex manifolds $(M_0, \scxs)$ and $(M_1, \scxs_1)$.

\begin{Theorem} \label{thm:surgery_theorem_alternative}
There is a topological Stein cobordism $(W, \zeta)$ from $(M_0, \zeta_0)$ to $(M_1, \zeta_1)$, i.e.~$(M_0, \zeta_0) \prec (M_1, \zeta_1)$,
if and only if there is a map $\alpha$ of fibrations over $BU$ 
\[ \xymatrix{ B^{q{-}1}_{\zeta_0} \ar[dr]_{\eta^{q{-}1}_{\zeta_0}} \ar[rr]^{\alpha} && 
B^{q{-}1}_{\zeta_1} \ar[dl]^{\eta^{q{-}1}_{\zeta_1}} \\ & BU } \]
and $\zeta_i$-compatible normal $(q{-}1)$-smoothings $ \bar \zeta_i \colon M_i \to B^{q{-}1}_{\zeta_i}$ 
such that
\[ \alpha_*([M_0, \bar \zeta_0]) = [M_1, \bar \zeta_1] \in 
\Omega_{2q{+}1}(B^{q{-}1}_{\zeta_1}; \eta^{q{-}1}_{\zeta_1}). \]
\end{Theorem}

\begin{proof}
Suppose that $\alpha_*([M_0, \bar \zeta_0]) = [M_1, \bar \zeta_1]$ and let $(W,\bar \zeta)$ be a $(B^{q{-}1}_{\scxs_1}, \eta^{q{-}1}_{\scxs1})$-nullcobordism of 
$(-M_0, -\alpha \circ \bscxs_0) \sqcup (M_1, \bscxs_1)$. 
Applying surgery below the middle dimension \cite[Proposition 2.6]{BCS2}, 
we can assume that $\bscxs \colon W \to B^{q{-}1}_{\scxs_1}$ is 
a $(q{+}1)$-equivalence. It follows from a result of Wall \cite[Theorem 2.18]{BCS2} that $W$ is built 
from $M_1$ by attaching handles of index $\ge q{+}1$ and dually that $W$ is obtained from $M_0$ 
by attaching handles of index at most $q{+}1$, verifying one direction of the equivalence. 


Conversely, if $(W, \zeta)$ is a topological Stein cobordism with boundary
$(-M_0, -\scxs_0) \sqcup (M_1, \scxs_1)$, then 
the universal properties of Postnikov factorizations 
\cite[Chapters 2 \& 5]{Baues77} mean that there
is a homotopy commutative diagram,
\[ \xymatrix{ -M_0 \ar[d]_{-\bscxs_0} \ar[r]^{i_0} & W \ar[d]^{\bscxs} & 
M_1 \ar[d]^{\bscxs_1} \ar[l]_{i_1} \\
B^{q{-}1}_{\scxs_0} \ar[dr] \ar[r]^{Bi_0} & B^{q{-}1}_{\scxs} \ar[d] & 
B^{q{-}1}_{\scxs_1} \ar[dl] \ar[l]_{Bi_1} \\
& BU&
,  } \]
where $\bscxs_0, \bscxs_1$ and $\bscxs$ are all $(q{-}1)$-smoothings and
for $j = 0, 1$, $i_j \colon M_j \to W$ are the inclusions and $Bi_j$
are the corresponding induced maps of complex normal $(q{-}1)$-types.
Since $W$ is obtained from $M_1$ by the addition of handles on index
$(q{+}1)$ or larger, the proof of \cite[Lemma 2.9 (3)]{BCS2}
shows that $Bi_1$ is an equivalence of complex normal $(q{-}1)$-types.
We then set $\alpha$ to be the following map of fibrations over $BU$:
\[ \alpha  := (Bi_1)^{-1} \circ Bi_0 \colon B^{q{-}1}_{\scxs_0} \to B^{q{-}1}_{\scxs_1},\]
where $(Bi_1)^{-1}$ denotes a homotopy inverse of $Bi_1$. By definition, $(W, \bscxs)$ is a $(B^{q{-}1}_{\scxs_1}, \eta^{q{-}1}_{\scxs_1})$-bordism
which gives $\alpha_*([M_0, \bscxs_0]) = [M_1, \bscxs_1]$.
%
%
%
\end{proof}

\begin{Remark}
Notice that \cite[Theorem~3.8]{BCS2} is a direct consequence of the above result (by taking
$\alpha ={\rm id}$). While \cite[Theorem~3.8]{BCS2} is symmetric for 
$(M_0, \zeta _0)$ and $(M_1, \zeta _1)$, in Theorem~\ref{thm:surgery_theorem_alternative} 
the direction
of the map $\alpha$ breaks this symmetry.
The above extension of our earlier result  was suggested by a question from 
Andy Wand in Nantes in September 2013.
\end{Remark}

\section{Maximal almost contact manifolds}
\label{sec:max}

In dimension three the Stein cobordism relation has several interesting
properties, one of which is that there are \emph{initial} elements: There
exists a contact manifold $(M_{min},\xi_{min})$ such that for any other contact 3-manifold
 $(M, \xi )$ we have 
$$(M_{min},\xi_{min}) \prec (M,\xi).$$ 
In fact, by \cite{EtHond} any overtwisted contact
structure on any manifold will do (see also \cite{GeigesZemisch}). On the other hand, in
high dimensions, i.e.\ for $\dim (M) = 2q{+}1 \ge 5$ there exist
\emph{final} almost contact elements. It is not clear whether such objects exist in
dimension three.

The next proposition provides a proof of
Theorem~\ref{thm:maximalelement} from the Introduction.

\begin{Proposition}\label{prop:final_elements}
In every dimension $2q{+} 1 \ge 5$ there exists an almost contact
manifold $(M_{{max}},\varphi_{{max}})$ so that for any almost contact manifold $(M, \varphi  )$ we have
$$(M ,\varphi  ) \prec (M_{{max}},\varphi_{{max}}).$$
 Moreover, in dimensions $5$ and $7$ we can take certain almost contact structures on the non{-}trivial sphere bundles over $S^2$ as final elements:
$$(M_{{max}},\varphi_{{max}}) = (S^3 \tilde \times S^2,\varphi_{{max}}) , \quad (M_{{max}},\varphi_{{max}}) = (S^5 \tilde \times S^2,\varphi_{{max}}),$$
where $\varphi_{{max}}$ is any almost contact structure whose first Chern class is primitive.
\end{Proposition}
\begin{proof}
For $(M_{{max}},\varphi_{{max}})$ we can take any almost contact $(2q{+}1)$-manifold where the corresponding 
stable complex manifold $(M_{{max}},\scxs_{max})$ has complex normal $(q{-}1)$-type $B^{q{-}1}_{\scxs_{max}}= BU$ and the map to $BU$ is just the identity. 
To construct such a manifold, we begin
with {\em any} stably complex $(2q{+}1)$-manifold $\scxs \colon M \to BU$ 
and apply surgery below the middle dimension 
\cite[Proposition 2.6]{BCS2} to obtain a stably complex manifold $(M_{max}, \zeta_{max})$ 
where $\zeta_{max} \colon M_{max} \to BU$ is a $q$-equivalence, which then has the desired complex normal $(q{-}1)$-type. 
We then take $\acs_{max}$, to be any almost contact structure which stabilizes to $\zeta_{max}$, which exists by \cite[Lemma 2.17]{BCS2}.

Now let $(M, \acs)$ be any almost contact $(2q{+}1)$-manifold with stable complex
structure $\zeta$, complex normal $(q{-}1)$-type $(B^{q{-}1}_\zeta, \eta^{q{-}1}_\zeta)$
and with $\zeta$-compatible normal $(q{-}1)$-smoothing $\bscxs \colon M \to B^{q{-}1}_\scxs$.
By definition, $\eta^{q{-}1}_\zeta \colon B^{q{-}1}_\zeta \to BU$ is a fibration and we set
$\alpha = \eta^{q{-}1}_\zeta$. Bordism of $BU$-manifolds is just
ordinary complex bordism, and by \cite[p.117]{Stong1} the
odd bordism groups $\Omega^U_{2q{+}1}$ are trivial, which implies that 
$\alpha_*([M, \bscxs]) = 0 = [M_{max}, \scxs_{max}]$.
By Theorem \ref{thm:surgery_theorem_alternative} it follows that $(M, \scxs) \prec (M_{max}, \scxs_{max})$
and due to the equivalence given in \eqref{eq:stable_and_unstable_Stein_cobordism} above we finally conclude that $(M, \acs) \prec (M_{max}, \acs_{max})$.

In dimensions $5,7$ one checks that the explicit manifolds stated in
the proposition have the correct complex normal $1$-
resp.\ $2$-types. For this, note that $\pi_1(BU) = \pi_3(BU) = 0,
\pi_2(BU) = \Z$ and the assumption that $c_1 $ is primitive ensures
that the second homotopy group of the associated type is $\Z$.
\end{proof}

\begin{Remark} \label{rem:maximal_is_not_unique}
The almost contact manifold $(M_{max}, \varphi_{max})$ is far from being unique.
Indeed, if $(M_{max}, \varphi_{max}) \prec (M', \varphi')$ then $(M', \varphi')$ is also maximal for the topological Stein cobordism relation.
For example, for any Stein fillable almost contact manifold $(M_0, \varphi_0)$,
$(M_{max} \# M_0, \varphi_{max} \# \varphi_0)$ is also maximal. Note also that 
$(M_{max}, \varphi_{max})$ is necessarily Stein fillable (that is, contains a Stein fillable
contact structure), shown by the Stein cobordism from, say, the standard contact sphere to 
$(M_{max}, \varphi_{max})$.
\end{Remark}

\begin{Remark} \label{rem:contact_max}
At the level of contact structures, it seems very unlikely that the 
analogue of Proposition \ref{prop:final_elements} holds.
Specifically, it seems unlikely that 
there is a single contact $(2n{+}1)$-manifold $(M_{max}, \xi_{max})$
such that for  every contact $(2n{+}1)$-manifold $(M, \xi)$ the manifold
 $(M_{max}, \xi_{max})$
is the out-going end of some 
Stein cobordism starting from $(M, \xi)$.
However by Theorem \ref{thm:h-principleforhandles} we have the following:
For any maximal almost contact manifold  $(M_{max}, \varphi_{max})$ of dimension $(2q{+}1) \geq 5$,
and for any 
contact $(2q{+}1)$-manifold $(M, \xi)$, $(M_{max}, \varphi_{max})$ admits 
a contact structure $\xi_{(M, \xi)}$
and a Stein cobordism from $(M, \xi)$ to $(M_{max}, \xi_{(M, \xi)})$.
\end{Remark}

A further interesting special case of the Stein cobordism relation
occurs for ``Calabi-Yau'' almost contact structures on $5$-manifolds and $7$-manifolds.

\begin{Proposition}\label{Calabi_Yau}
Let $n = 5$ or $7$ and let $(M,\varphi)$ be an almost contact $n$-manifold 
such that $c_1(\varphi) = 0$. Then
$$(M,\varphi) \prec (S^n,\varphi_{std}),$$ where $\varphi_{std}$ denotes
the almost contact structure underlying the standard contact structure on $S^n$ and the Stein cobordism
$(W,J)$ can be assumed to have $c_1(J) = 0$.
\end{Proposition}

\begin{proof}
Let $n = 2q{+}1$, so that $q = 2$ or $3$ and let $\scxs$ be the stabilization of the almost
contact structure
$\varphi$.
Since $c_1(\varphi) = 0$, the complex normal $(q{-}1)$-type of $(M, \scxs)$ 
factors through
$BSU \to BU$ by \cite[Lemma 2.22 (ii), Lemma 2.23]{BCS2}.
Since the complex normal $(q{-}1)$-type of $(S^n,\varphi_{std})$ is $BSU \to BU$,
the lemma follows from Theorem
\ref{thm:surgery_theorem_alternative} and the fact that 
$\Omega_5^{SU} = \Omega_7^{SU}= 0$; see \cite[p.\,248]{Stong1}.
\end{proof}
\begin{Remark}
 In contrast to the 3-dimensional case, there can be no minimal
 elements with respect to the Stein cobordism relation in dimension at
 least 5. For example if $(M_0, \varphi _0) \prec (M_1, \varphi _1)$ and $c_1(\varphi_1) = 0$, then the
 fact that $c_1(\varphi_1) = 0$ implies that $c_1(\varphi_0) =0$. To see this note that a topological Stein cobordism $(W,J)$ from $M_0$ to $M_1$ is obtained from $M_1$ by attaching high index
 handles, which means that $c_1(J) = 0$. But this would imply that an
 initial element $(M_{min},\varphi_{min}) $ must have $c_1(\varphi_{min})
 = 0$. A similar argument shows that for certain choices of $M_1$ the
 fact that $c^q_1(\varphi_1) \neq 0$ implies $c^q_1(\varphi_{min}) \neq 0$
 (cf.\ \cite[Proposition 6.2]{BCS2}). For this note that the inclusion
 of $M_1$ into $W$ gives an isomorphism on fundamental groups.
Suppose that $\beta \in H^1(B \pi_1(W))$ is a class such that when restricted to $M_1$
\[
p^*_1(\beta) \cup c^q_1(J)\neq 0, 
\] 
where $p_1\colon W \longrightarrow B \pi_1(W)$ is the classifying map
of the universal cover of $W$.
It follows that the restriction of $p^*_1(\beta) \cup c^q_1(J)\neq 0$ to $M_0=M_{min}$ is also non{-}zero and hence that $c_1(\varphi_{min}) = c_1(J)|_{M_{min}} \neq 0$.

In fact, these sorts of arguments show that there can be
no initial elements even if one forgets about the almost contact
structures, and one simply considers Pontrjagin classes rather than Chern
classes.
\end{Remark}

\section{Stein fillability and connected sums} 
\label{subsec:stein_fillability_and_connected_sums}

The connected sum of Stein fillable manifolds is again
Stein fillable, since adding a one-handle to the Stein fillings can
be done in a way that is compatible with Stein structures. Eliashberg \cite[Section 8]{Eliashberg90} has shown that the 
converse of this statement holds for 3-manifolds: a connected sum
of $3$-manifolds is Stein fillable if and only if both summands
are. In addition, the Stein fillable structures on the components
can be chosen so that their connected sum is isotopic to the given Stein fillable
structure on the connected sum.
(For a detailed description of the contact connected sum construction
see \cite[p. 301-302]{Geiges08}.) 
Extensions of some aspects of the above result of Eliashberg
to higher dimensions are
given in \cite{GNW}.

In this section we prove Theorem~\ref{thm:mainconnsum}, which shows 
that in higher dimensions the summands of a Stein fillable almost
contact manifold in a connected sum decomposition are not necessarily
Stein fillable.  To make this statement precise, we note that the summands are
well-defined only up to \emph{almost diffeomorphism}, that is, up to
connected sum with homotopy spheres.  Indeed, for $M_1\# M_2$ the
manifolds $M_1-{\rm {int}}\ D^n$ and $M_2- {\rm {int}}\ D^n$ can be
turned into closed manifolds (by gluing back $D^n$) in many different
ways, differing by connected sums with homotopy spheres.  Below we 
find examples  of almost contact connected sums
where the summands are not Stein fillable, even after the addition
of homotopy spheres.

For the proof of Theorem~\ref{thm:mainconnsum} we examine the Stein
fillability of certain almost contact structures on the unit cotangent
bundle $ST^*S^{2k{+}1}$ of the $(2k{+}1)$-sphere  $S^{2k{+}1}$.  We first need to
establish some preliminary results.  Lemma \ref{lem:cotangent} is a
small elaboration of a theorem of Milnor and Spanier about the
topology of $ST^*S^{2k{+}1}$.
Proposition \ref{prop:filling_cotangent_spheres} gives a 
description of the topology of possible Stein fillings of manifolds almost
diffeomorphic to $ST^*S^{2k{+}1}$, which may be of independent interest. 
Finally, Lemmas \ref{lem:c_k} and \ref{lem:non_filling_cotangent} show that $ST^*S^{2k{+}1}$
admits an almost contact structure $\varphi$ which is not Stein fillable,
provided $k \geq 5$ is odd.

\begin{Lemma}[cf.~{\cite[Theorem 2]{Milnor&Spanier}}] \label{lem:cotangent}
There is a map $f \colon ST^*S^{2k{+}1} \to S^{2k}$ such that the induced homomorphism
$f_* \colon H_{2k}(ST^*S^{2k{+}1}) \to H_{2k}(S^{2k})$ is an isomorphism if and only if $k = 0, 1$ or $ 3$.
\end{Lemma} 

\begin{proof}
Let $\pi \colon ST^*S^{2k{+}1} \to S^{2k{+}1}$ be the bundle projection
of the unit cotangent bundle of $S^{2k{+}1}$. 
If we consider $\pi$ merely as a spherical fibration,
then a map $f$ as in the statement of the lemma exists if and only if $\pi$ is trivial as
a spherical fibration.  This is because the product map
\[ f \times \pi \colon ST^*S^{2k{+}1} \to S^{2k} \times S^{2k{+}1} \]
is a homology isomorphism, and so by Whitehead's theorem a homotopy
equivalence, and this gives a fibre homotopy trivialisation of $\pi$.
By \cite[Theorem~2]{Milnor&Spanier}, the bundle projection $\pi$ is
trivial as a spherical fibration if and only if $k = 0, 1$ or $3$.
\end{proof}

The proof of the following proposition uses handle cancelling and the
Whitney trick, which are familiar from the proof of the $h$-cobordism
theorem in higher dimensions. For details concerning these
constructions we refer to \cite{Milnor65, Whitney44}.
\begin{Proposition}\label{prop:filling_cotangent_spheres}
Let $k \geq 1$ and let $M = (ST^*S^{2k{+}1}) \# \Sigma _0$ be the
connected sum of $ST^*S^{2k{+}1}$ with some homotopy sphere $\Sigma
_0$. Choose some almost contact structure $\varphi$ on $M$ and let
$W^{4k+2}$ be the smooth manifold underlying a Stein filling of $(M,
\acs)$.  Then $W$ decomposes as a boundary connected sum
\[ W =W_l\thinspace \natural W_\Sigma,\]
%
where $H_{2k}(W_l) = \Z /l\Z$ and $W_\Sigma$ is a $2k$-connected filling of some homotopy
sphere $\Sigma$. 
Moreover we have the following possibilities for the topology of $W_l$:
\begin{enumerate}

\item If $l > 1$, then $W_l$ has a handle decomposition with precisely
  one handle of index $2k$ and two handles of index $2k{+}1$;

\item If $l = 1$, then $W_1 \cong DT^*S^{2k{+}1}$;

\item If $l = 0$, then we must have $k = 1, 3$ and $W_0 \cong D^{2k{+}2} \times S^{2k}$.

\end{enumerate}
In particular, $ST^*S^{2k{+}1}$ admits a subcritical filling if and only if $k = 1,3$.
\end{Proposition}

\begin{proof}
First note that $M=\del W$ is $(2k{-}1)$-connected and that $H_*(\del W)
\cong H_*(S^{2k} \times S^{2k{+}1})$.  It follows that any Stein filling
of $M$ is $(2k{-}1)$-connected and hence by the Hurewicz Theorem the map
$\pi_{2k{+}1}(W) \longrightarrow H_{2k{+}1}(W)$ is surjective.  The long
exact sequence of the pair
%
$$ H_{2k{+}1}(\partial W)\stackrel{\iota} \longrightarrow H_{2k{+}1}(W)
\longrightarrow H_{2k{+}1}(W,\partial W)\longrightarrow
H_{2k}(\partial W) \longrightarrow H_{2k}(W) \longrightarrow 0$$
yields that $H_{2k}(W) = \Z / l\Z$ is cyclic and we also know that
$H_{2k{+}1}(W)$ is torsion free since $W$ admits a handle decomposition without handles of index greater than $2k+2$ by assumption. Moreover, the intersection pairing is unimodular on a
complement $H$ to $im(\iota) \subseteq H_{2k{+}1}(W)$.  We let
$\{x_1,\ldots, x_{2r}\}$ be a symplectic basis for $H$ consisting of
primitive elements, which can in turn be represented by spheres. Since
the target is simply connected, we can use the Whitney trick to find
embedded representatives in the interior of $W$. We can furthermore
assume that the geometric intersection numbers of these spheres agree
with their algebraic intersection numbers.

Thus we have a configuration of embedded $(2k{+}1)$-dimensional spheres $\{S_1,\ldots, S_{2r}\}$ 
having a regular neighbourhood $N$ whose boundary is a homotopy sphere $\Sigma$.
It follows that $W$ decomposes as a boundary connected sum
$$W =W_l\thinspace  \natural \thinspace W_\Sigma,$$
where $W_\Sigma \cong N$ is $2k$-connected. This boundary connected sum is obtained by choosing an embedded path $\gamma$ from $\partial N$ to $\partial W$ and removing $N$ along with a tubular neighbourhood of $\gamma$. Note that in this decomposition $\partial W_l = M \# (-\Sigma)$.  Applying Mayer-Vietoris, we conclude that $H_{2k}(W_l) \cong \Z / l\Z$.
This proves the first part of the proposition.

We now consider the topology of $W_l$ and prove the remainder of the
proposition.  $W_l$ is $(2k{-}1)$-connected, hence it follows
by handle cancelling that $W_l$ is obtained by attaching handles of
index at least $2k$ to $ST^*S^{2k{+}1} \# \Sigma_0 \# (-\Sigma)$. Turning this handle
decomposition upside down gives a handle decomposition with handles of
index at most $2k+2$.  Then by further cancellation of handles we can
find a handle decomposition with at most one handle of index $2k$ and
at most two handles of index $2k{+}1$, where we use the fact that $W_l$
is simply connected.  This proves case $(1)$.

If $l = 1$, then the handle decomposition of $W_l$ reduces further to
contain a single $(2k{+}1)$-handle, and $W_1$ is diffeomorphic to a
linear $D^{2k{+}1}$-bundle over $S^{2k{+}1}$.  This bundle must be stably
trivial, otherwise $M$ will have a non{-}trivial stable tangent bundle.
Moreover, by analysing the homotopy long exact sequence of the fibration
\[
SO(2k{+}1) \to SO(2k{+}2) \to S^{2k{+}1},
\]
one sees that any stably trivial bundle over $S^{2k{+}1}$ is either
trivial or isomorphic to the unit tangent bundle of the sphere. If $k
\neq 1, 3$ then $W_1$ cannot be the total space of the trivial bundle,
because this would give rise to the existence of a map $f \colon M \to
S^{2k}$ as in Lemma \ref{lem:cotangent} which is impossible.  On the
other hand, if $k = 1, 3$, then the tangent bundle of $S^{2k{+}1}$ is trivial.
In both cases we conclude that there is a diffeomorphism $W_1 \cong
DT^*S^{2k{+}1}$, which proves case $(2)$.

If $l = 0$, then $H_{2k}(W) = H_{2k}(W_0) =\Z$. 
It follows that the handle
decomposition of $W_0$ has just one $2k$-handle and hence $W_0$ is
diffeomorphic to a linear $D^{2k{+}2}$-bundle over $S^{2k}$.  This
bundle is stable and must be stably trivial, otherwise the tangent bundle of
$M$ would not be stably trivial.  We conclude that $W_0$ is
diffeomorphic to $D^{2k{+}2} \times S^{2k}$.
\end{proof}


\begin{Remark} \label{rem:6d-fillings}
Although not the focus of this work, the topology of Stein fillings,
as opposed to their boundaries, is of independent interest and may be
relevent to certain computations in contact homology.  In the case that $k = 1$, Proposition \ref{prop:filling_cotangent_spheres} determines the
smooth manifolds underlying Stein fillings of almost contact
structures on $S^2 \times S^3$.  In this dimension there are no exotic
$5$-spheres \cite{Kervaire-Milnor63} and by \cite[Theorem
  6.2]{Smale62} the manifold $W_\Sigma$ is diffeomorphic to $\#_r(S^3
\times S^3)-{\rm Int}(D^6)$.  Concerning the manifolds $W_l$, we
conjecture that they are classified up to diffeomorphism by $l$.  If
this is correct, then topological Stein fillings $(W, J)$ of $S^2
\times S^3$ are classified up to stably complex diffeomorphism by
their integral homology groups $H_*(W)$ along with their first Chern class 
$c_1(J) \in H^2(W)$.
\end{Remark}

\begin{Lemma} \label{lem:c_k}
If $k$ is odd, there exist almost contact structures $\varphi$ on
$ST^*S^{2k{+}1}$ with non{-}zero $k^{th}$-Chern class, $0\neq c_k(\varphi)
\in H^{2k}(ST^*S^{2k{+}1}) \cong \Z$.
\end{Lemma}
 
\begin{proof}
Let $\varphi_{can}$ be the standard almost
contact structure underlying the canonical contact structure
on $M = ST^*S^{2k{+}1}$,
and let $\scxs_{can}$ be the stable complex structure determined by
$\varphi_{can}$.  Since $\scxs_{can}$ extends over $DT^*S^{2k{+}1}$, we
have that $c_k(\varphi_{can}) = c_k(\scxs_{can}) = 0$.  
To find a stable complex structure $\scxs$ with $c_k(\scxs) \neq 0$, we recall that
the group $[M, SO/U]$ acts freely and transitively on the set of homotopy classes of 
stable complex structures on $M$.  Now if we let $M^\bullet := M \setminus B^{4k+1}$
be the manifold obtained by removing a ball, then there is a homotopy equivalence
$M^\bullet \simeq S^{2k} \vee S^{2k+1}$ and hence 
\begin{equation} \label{eq:M_to_SO/U}
[M^\bullet, SO/U] \cong \pi_{2k}(SO/U) \oplus \pi_{2k{+}1}(SO/U).
\end{equation}
Since $k$ is odd, $\pi_{4k}(SO/U) = 0$ \cite{Bott59}, and so there is no obstruction to extending
a map $M^\bullet \to SO/U$ to a map $M \to SO/U$.  Hence the restiction
map $[M, SO/U] \to [M^\bullet, SO/U]$ is onto.
Again using that $k$ is odd, the boundary map, $\pi_{2k}(SO/U) \to \pi_{2k{-}1}(U)$, 
in the homotopy long exact sequence of the fibration
$U \to SO \to SO/U$ is non{-}zero \cite{Bott59}.
Since $\pi_{2k{-}1}(U)$ classifies stable unitary bundles over $S^{2k}$
which are in turn classified by their $k^{th}$ Chern class
\cite[Proposition 9.1]{Husemoller94}, it follows from \eqref{eq:M_to_SO/U} and the discussion above
that we can choose
$\psi \in [M, SO/U]$ such that $\scxs : = \scxs_{can} + \psi$ has
$c_k(\scxs) \neq 0$.  By \cite[Lemma 2.17]{BCS2}, we know that
$\scxs$ destabilizes to an almost contact structure
$\varphi$, which then also has $c_k(\varphi) \neq 0$.
\end{proof} 
 
\begin{Lemma}\label{lem:non_filling_cotangent}
Let $\varphi$ be an almost contact structure on $M = ST^*S^{2k{+}1}$ with $c_k(\varphi) \neq 0$ 
and let $(\Sigma,\varphi_{\Sigma})$ be any almost contact homotopy
sphere. If $k \neq 1, 3$, then neither $(M \# \Sigma, \varphi \# \varphi_{\Sigma})$
nor  $\bigl( -(M \# \Sigma), -(\varphi \# \varphi_{\Sigma}) \bigr)$ 
is Stein fillable.
\end{Lemma}

\begin{proof}
Suppose that $(W, J)$ is a Stein filling of 
$(M\#\Sigma,\varphi\#\varphi_{\Sigma})$.
Since $c_k(\varphi\#\varphi_{\Sigma}) = c_k(\varphi)$ is non{-}zero and pulls back from
$c_k(J) \in H^{2k}(W)$, we conclude that $H^{2k}(W)$ is infinite.
Now by Proposition \ref{prop:filling_cotangent_spheres}, this can only happen if
$k = 1, 3$.  Since we assumed $k \neq 1, 3$, no such Stein filling $(W, J)$ can exist.
For the reversed orientation, we use \cite[Propostion 6.7]{BCS2} which states that
$(M \# \Sigma, \varphi \# \varphi_{\Sigma})$ is Stein fillable if and only if 
$\bigl( -(M \# \Sigma), -(\varphi \# \varphi_{\Sigma}) \bigr)$ is Stein fillable.
\end{proof}

\begin{proof}[Proof of Theorem \ref{thm:mainconnsum}]
We let $k \geq 5$ be odd and set $(M,\varphi) = (ST^*S^{2k{+}1},\varphi)$, where we have $c_k(\varphi) \neq 0$.
By Lemma \ref{lem:c_k} such almost contact structures always exist. 
Now, by Lemma \ref{lem:non_filling_cotangent}, for any almost contact homotopy sphere
$(\Sigma, \varphi_{\Sigma})$ neither 
$(M \# \Sigma, \varphi \# \varphi_{\Sigma})$
nor  $\bigl( -(M \# \Sigma), -(\varphi \# \varphi_{\Sigma}) \bigr)$ 
is Stein fillable.

On the other hand, we let $M^{\bullet} = M \setminus B^{4k{+}1}$ 
and we let $\varphi^{\bullet}$ be the induced almost contact structure on $M^\bullet$. 
We set $W^{4k+2} := M^\bullet \times [0,1]$,
which has a natural almost complex structure $\varphi^\bullet \times [0, 1]$ induced by
$\varphi^{\bullet}$. Moreover, the smoothened boundary of $W$ with the
induced almost contact structure is precisely 
$(M,\varphi) \# (-M, -\varphi)$. 
Since $M$ is $(2k{-}1)$-connected, it follows from a theorem of Smale \cite[Theorem C]{Smale61}, that $M^\bullet$ admits a handle decomposition with handles of index less than or equal to $2k{+}1$.
Since any handle decomposition on $M^\bullet$ gives rise to one on $M^\bullet \times I$ with handles of the same index, we have that $W = M^\bullet \times [0, 1]$  admits a handle decomposition with handles of index less than or equal to $2k{+}1$. 
Consequently $(W, \varphi^\bullet \times [0, 1])$
is a topological Stein filling of $(M, \varphi) \# (-M, -\varphi)$, 
which is then admits a Stein fillable contact representative 
$\xi _{M\# (-M)}$ by Theorem \ref{thm:h-principle}.
\end{proof}

\begin{Remark}
Notice that the non{-}fillability of $(ST^*S^{2k{+}1}, \acs)$ in Theorem
\ref{thm:mainconnsum} arises from the choice of the almost contact
structure $\acs$, since $ST^*S^{2k{+}1}$ does admit Stein fillable
contact structures.  In \cite{BCS3} we shall prove a stronger version
of Theorem \ref{thm:mainconnsum} which asserts the existence of
$(4k{-}1)$-connected closed smooth $(8k{+}1)$-manifolds $M$, such that $M$
(and $M\# \Sigma$ for any homotopy sphere $\Sigma$)
admits no Stein fillable almost contact structure at all, but $M \# (-M)$ is 
Stein fillable.
We would like to point out that our result is on the almost contact level:
we do not claim that the Stein fillable contact structure $\xi _{M\# (-M)}$ on $M\# (-M)$ 
found in the proof of Theorem~\ref{thm:mainconnsum}
(representing the almost contact structure $\varphi \# (-\varphi )$) 
can be given
as a connected sum 
$\xi _+\# \xi _-$ where $\xi _{\pm}$ is a contact structure on $\pm M$. 
\end{Remark}


We now turn to dimension $5$ and prove Theorem \ref{thm:5fill},
restated below as Theorem \ref{thm:5fill-main}.
Notice that in 
dimension five the connected sum $M_0\# M_1$ determines the diffeomorphism 
type of its components $M_0$ and $M_1$, since there are no exotic $5$-spheres.
The following lemma extends this statement to almost contact $5$-manifolds.

\begin{Lemma} \label{lem:5d-almost-contact-disconnected-sum}
Let $\acs$ be an almost contact stucture on the connected sum of $5$-manifolds
$M_0$ and $M_1$.  Then there are, up to homotopy unique, 
almost contact structures $\acs_0$ on $M_0$ and $\acs_1$ on $M_1$, such that $(M_0 \# M_1, 
\acs) = (M_0 \# M_1, \acs_0 \# \acs_1)$.
\end{Lemma}

\begin{proof}
For $i = 0, 1$, let $M_i^\bullet := M_i - {\rm int}(D^5) \subset M_0 \# M_1$ be the 
punctured copy of $M_i$ contained in the connected sum.  We define $\acs_i|_{M_i^\bullet} := \acs|_{M_i^\bullet}$.
It remains to show that there is a unique extension of $\acs_i|_{M_i^\bullet}$ to an almost contact structure 
on $M_i$.  Now the obstruction to extension lies in $\pi_4(SO(5)/U(2))$ and the obstruction to uniqueness 
lies in $\pi_5(SO(5)/U(2))$.  By \cite{Massey61}, we have 
$\pi_4(SO(5)/U(2)) = \pi_5(SO(5)/U(2)) = 0$, which concludes the proof.
\end{proof}
\noindent With the aid of this lemma we have
\begin{Theorem}\label{thm:5fill-main}
Let $(M, \acs) = (M_0 \# M_1, \acs_0 \# \acs_1)$ be a Stein fillable almost contact $5$-manifold.
Assume that either $c_1(\acs) = 0$ or that
\[ c_1(\acs)(\pi_2(M_0)) = c_1(\acs)(\pi_2(M_1))  = \Z = \pi_2(BU).\]
Then both $(M_0, \acs_0)$ and $(M_1, \acs_1)$ are Stein fillable.
\end{Theorem}

%
%
%
%

\begin{proof}
Let $\scxs, \scxs_0$ and $\scxs_1$ be the stable complex structures determined by $\acs, \acs_0$
and $\acs_1$ respectively.  After stabilizing we have $(M, \scxs) = (M_0, \scxs_0) \# (M_1, \scxs_1)$.
Also, let $K_i = K(\pi_1(M_i), 1)$ so that $K(\pi_1(M), 1) = K_0 \vee K_1$.
Under the assumptions of the proposition, \cite[Lemma~2.13]{BCS2} implies that the complex normal
$1$-type of $(M, \scxs)$ is given by
\[ 
(B^1_\scxs, \eta^1_\scxs) \cong (BSU \times 
\left( K_0 \vee K_1 \right), 
{\rm pr}_{BSU}) 
\]
if $c_1(\acs) = 0$ and
\[ 
(B^1_\scxs, \eta^1_\scxs) \cong (BU \times 
\left( K_0 \vee K_1 \right), 
{\rm pr}_{BU}) 
\]
if $c_1(\acs)(\pi_2(M_0)) =c_1(\acs)(\pi_2(M_1))= \pi_2(BU)  = \mathbb{Z}$. 
Since $\eta^1_\scxs$ is the projection to $BSU$
resp.\ $BU$, there is a canonical isomorphism 
$\Omega_5(B^1_\zeta; \eta^1_\zeta) \cong \Omega_5^G(K_0 \vee K_1)$ for $G = U$ or $SU$,
which we use in the remainder of the proof.

Let us now assume that $c_1(\acs) = 0$.
The argument in the other cases is formally the same, and is given by replacing
the map $BSU \to BU$ by ${\rm Id} \colon BU \to BU$. 
Let $\bscxs \colon M \to B^1_\scxs$ be a $\scxs$-compatible normal $1$-smoothing.
We first consider the product smoothing 
$\bscxs \circ {\rm pr}_M \colon M \times [0, 1] \to M \to B^1_\scxs$ and attach
a 6-dimensional $5$-handle to $M \times \{1\}$ along the connect sum locus
$S^4 \times [0, 1] \subset M = M_0 \# M_1$ to obtain a bordism $W$.  
Since $\pi_4(SO/SU) = 0$, there is
no obstruction to making this a stably complex $5$-handle (see \cite[Section 2.3]{BCS2}),
and indeed there is no obstruction to extending $\bscxs \circ {\rm pr}_M$ to a normal smoothing
$\bscxs'_{W} \colon W \to B^1_\scxs$.
Consider the map $\bscxs_W$ given by taking the composition
of $\bscxs'_{W}$ with the collapsing map induced by the
wedge sum:

\[ W \stackrel{\bscxs'_{W}} \longrightarrow BSU \times (K_0 \vee K_1) \stackrel{{\rm col}} \longrightarrow BSU \times K_1.\]
This is a normal map, and setting $\bscxs_i$ to be the restriction of $\zeta_W$ to $M_i$ we see that the bordism
$(W, \bscxs_W)$ gives the equality
\[ [M, {\rm col} \circ \bscxs] = [M_0, \bscxs_0] + [M_1, \bscxs_1] \in \Omega_5^{SU}(K_1).\]
Now, by Theorem \ref{thm:stein}, $[M, \bscxs] = 0$ since $(M, \acs)$ is Stein fillable, and
consequently the bordism class $[M, {\rm col} \circ \bscxs]$ is trivial too.  Moreover, since
the composition ${\rm pr}_{K_1} \circ \bscxs_0 \colon M_0 \to K_1$ is null-homotopic and the bordism group $\Omega^{SU}_5 = \Omega_5^{U} = 0$ according to \cite[p.\,248]{Stong1}, it follows that $[M_0, \bscxs_0] = 0$. 
Hence the bordism class $[M_1, \bscxs_1]$ is trivial. Since the map
$$\bscxs_1 \colon M_1 \to BSU \times K_1$$
 is a 
$\scxs_1$-compatible normal $1$-smoothing, Theorem \ref{thm:stein} implies that $(M_1, \acs_1)$ is Stein fillable. The same argument \emph{mutatis mutandis} shows that
 $(M_0, \acs_0)$ is Stein fillable as well.
\end{proof}

\begin{Remark}
We point out that in dimension five the method in the proof of 
Theorem \ref{thm:5fill-main} does not ascend to 
give control over contact structures.  That is, if the almost contact manifolds
$(M_0, \acs_0)$ and $(M_1, \acs_1)$ are induced from contact manifolds $(M_0, \xi_0)$
and $(M_1, \xi_1)$, and even if we know that $(M_0 \# M_1, \xi_0 \# \xi_1)$ is 
Stein fillable, then in contrast to the situation in dimension 3,
we cannot conclude that $(M_0, \xi_0)$ and $(M_1, \xi_1)$
are Stein fillable.
%
\end{Remark}
\begin{Remark}
Note that in the proof of Theorem \ref{thm:5fill-main} involved constructing a nullbordism of each component of the connected sum $M_0 \# M_1$ by first adding a $5$-handle and then capping off two of the resulting boundary components. This bordism is thus far from having the correct homotopy type, but applying surgery below the middle dimension as in the proof of Theorem \ref{thm:stein} has the virtue of remedying this.
\end{Remark}
\section{Non{-}fillable almost contact structures on highly connected manifolds} 
\label{sec:nonst-acs}
In dimensions congruent to $7$~mod~$8$, the isomorphism
$\pi_{8k{-}1}(SO/U) \cong \Z_2$ means that there are precisely two homotopy classes of stable almost contact 
structures on $S^{8k{-}1}$.
One of these homotopy classes, denoted $\scxs_{std}$, bounds over $D^{8k}$ and is thus
Stein fillable.  Let us call the other stable almost contact
structure on $S^{8k{-}1}$ \emph{exotic} and denote it by
$\scxs_{ex}$.  By \cite[Theorem~1.3]{BCS2} we know that
$(S^7, \scxs_{ex})$ is Stein fillable.  Indeed, (according to
Theorem~\ref{thm:Yang+} below) the quaternionic projective plane ${\mathbb {HP}}^2$ 
admits no almost complex structure, but if we puncture it, then as the Hopf
$D^4$-bundle over $S^4$ it does. Thus the punctured ${\mathbb {HP}}^2$
provides a filling of $(S^7, \scxs_{ex})$ which admits a
Stein structure, inducing a Stein fillable contact structure on $S^7$
that stabilizes to $\zeta _{ex}$.  In higher dimensions, however, we have the following result (which 
corresponds to Theorem~\ref{thm:wackyspheres} from the Introduction):


\begin{Theorem}\label{thm:S8k{-}1}
The exotic stable complex structure $\zeta _{ex}$ on $S^{8k{-}1}$
cannot be represented by a Stein fillable contact structure once
$k\geq 2$.
\end{Theorem}
Before giving the proof of this result, we derive
Corollary~\ref{cor:wacky} as a simple consequence.

\begin{proof}[Proof of Corollary~\ref{cor:wacky}]
Let $\varphi$ be an almost contact structure on the 
$(4k{-}2)$-connected oriented $(8k{-}1)$-manifold $M$.
If $(M, \acs)$ is not Stein fillable, we are done.  If $(M, \acs)$ is
Stein fillable, let $\scxs$ be the stable complex structure determined by $\acs$
and observe that the complex normal $(4k{-}2)$-type of $(M, \scxs)$ is $(BU\an{4k}, \pi_{4k})$
from Example \ref{exa:normal-type}, with associated bordism groups $\Omega_*^{U\an{4k{-}1}}$.
Let $\bscxs \colon M \to B^{4k{-}2}_\scxs$ be a
$(4k{-}2)$-smoothing.  
By Theorem \ref{thm:stein} we have $[M, \bscxs]
= 0 \in \Omega_{8k{-}1}^{U\an{4k{-}1}}$.
The connected sum 
$(M, \bscxs \# \bscxs_{ex}) : = (M, \bscxs) \# (S^{8k{-}1}, \bscxs_{ex})$ 
does not change the complex
normal $(4k{-}2)$-type, and
\[ 
[M, \bscxs \# \bscxs_{ex}]  = [M, \bscxs] + [S^{8k{-}1}, \bscxs_{ex}] 
\in  \Omega_{8k{-}1}^{U\an{4k{-}1}}.
\]
%
By Theorem~\ref{thm:S8k{-}1} the stable complex manifold
$(S^{8k{-}1}, \zeta _{ex})$ is not
Stein fillable, and so by Theorem \ref{thm:stein} $[S^{8k{-}1}, \bscxs_{ex}] \neq 0$,
since $(BU\an{4k}, \pi_{4k})$ is the complex normal $(4k{-}2)$-type of $(S^{8k-1}, \bscxs_{ex})$.
It follows that $[M, \bscxs \# \bscxs_{ex}] \neq 0$ and consequently $(M, \scxs \# \scxs_{ex})$ is not Stein fillable by Theorem \ref{thm:stein}, 
since $(BU\an{4k}, \pi_{4k})$ is the complex normal $(4k{-}2)$-type of $(M, \scxs \# \scxs_{ex})$.
\end{proof}

\begin{Remark} \label{rem:Contact_existence}
According to \cite[Proposition 6 (vi)]{Geiges97},
the hypothesis of Corollary \ref{cor:wacky} that 
the $(4k{-}2)$-connected $(8k{-}1)$-manifold $M$ admit
an almost contact structure 
is equivalent to assuming that $\im(\tau_{M*}) \subseteq F_*(\pi_{4k}(BU))$. 
In Theorem~\ref{thm:Yang+}\,(3) below, 
we prove that provided $k\geq2$, the same condition is a necessary and sufficient condition for a $(4k{-}1)$-connected $8k$-manifold  to admit a stable complex structure.
\end{Remark}

\begin{Remark} \label{rem:Geiges-obstruction}
The stable almost contact structures $\scxs$ and $\scxs \# \scxs_{ex}$ appearing 
in the proof of Corollary~\ref{cor:wacky} differ by precisely the 
{\em ``top-dimensional $\Z/2$-obstruction to stable homotopy of almost contact structures''}
identified by Geiges in \cite[Theorem 4 (2b)]{Geiges97}.  

\end{Remark}

\begin{Remark}
By \cite{BEM} the stable almost contact structure found in Corollary~\ref{cor:wacky} (as any stable complex, or even almost contact structure)
can be represented by 
an overtwisted contact structure; according to Theorem \ref{thm:S8k{-}1} the stable almost contact structure $\zeta _{ex}$ on 
$S^{8k-1}$ as well as the stable contact structure found by Corollary~\ref{cor:wacky} 
cannot be represented
by Stein fillable contact structures. It would be interesting to see if these particular stable complex structures
admit fillable or tight (that is, not overtwisted) contact representatives.
 \end{Remark}

\subsection{Almost complex $8k$-manifolds}
\label{subsec:AlmostComplex}
In order to prove Theorem \ref{thm:S8k{-}1}, we first improve a theorem
of Yang \cite{Yang12} determining which smooth closed oriented
$(4k{-}1)$-connected $8k$-manifolds admit stable complex structures.


Let $Y$ be a smooth closed oriented $(4k{-}1)$-connected $8k$-manifold, 
let $\tau_Y \colon Y \to BSO$ classify the stable tangent bundle 
of $Y$ and let
\[ \tau_{Y*} \colon \pi_{4k}(Y) \to \pi_{4k}(BSO) \]
be the induced homomorphism.  
If $Y$ admits a stably complex structure, then $\tau_Y$ factors through $F \colon BU \to BSO$
and in this case ${\im(\tau_{Y*})} \subseteq F_*(\pi_{4k}(BU)) \subseteq \pi_{4k}(BSO)$.
Now let $B_i$ denote the $i^{th}$ Bernoulli number (where we
use the topological indexing and sign conventions, in particular,
$B_1=\frac{1}{6}$ and $B_2=\frac{1}{30}$).  The following theorem
is a straightforward reformulation of (2) and (3) of~\cite[Theorem~1]{Yang12}.
\begin{Theorem}[cf.~{\cite[Theorem 1, (2) \& (3)]{Yang12}}] \label{thm:Yang}
A smooth closed oriented $(4k{-}1)$-connected $8k$-manifold $Y$ with signature
$\sigma_Y$ admits a stable complex structure if and only if,
\begin{enumerate}
\item $k$ is odd and $\left(\frac{B_{2k}+B_{k}}{B_{2k}B_{k}} \cdot
  \frac{1}{2^{4k{-}2}}\right)\sigma_Y \equiv 0~mod~2$, or
\item $k$ is even, $\im(\tau_{Y*}) \subseteq F_*(\pi_{4k}(BU))$ and
  $\left(\frac{B_{2k}-B_{k}}{B_{2k}B_{k}} \cdot \frac{4k}{2^{4k}}\right)\sigma_Y
  \equiv 0~mod~2$.\qed
\end{enumerate}
\end{Theorem}

\begin{Remark} \label{rem:Yang_and_us}
Using the Hurewicz isomorphism $\pi_{4k}(Y) \cong H_{4k}(Y)$ and 
the Universal Coefficient Theorem, we regard the homomorphism
$\tau_{Y*} \colon \pi_{4k}(Y) \to \pi_{4k}(BSO) = \Z$ as a cohomology class 
$\tau_{Y*} \in H^{4k}(Y)$, which Yang denotes by $\nu$.
When $k$ is even, $F_*(\pi_{4k}(BU)) \subset \pi_{4k}(BSO)$ is the subgroup of index two 
\cite{Bott59}, and so the condition $\im(\tau_{Y*}) \subseteq F_*(\pi_{4k}(BU))$ is equivalent 
to the condition that $\tau_{Y*}$ vanishes mod~$2$, which is the condition Yang uses.
\end{Remark}

The following result, Theorem \ref{thm:YangIntro} from the Introduction,
simplifies Yang's theorem by removing the assumptions involving Bernoulli numbers from its statement.
\begin{Theorem} \label{thm:Yang+}
A smooth closed oriented $(4k{-}1)$-connected $8k$-manifold\,\,\,$Y$\! admits
a stable almost complex structure if and only if
\begin{enumerate}
\item $k\geq 3$ is odd, or
\item  $k=1$ and the  signature $\sigma_Y$ of $Y$  is even, or
\item $k$ is even and  $\im(\tau_{Y*}) \subseteq F_*(\pi_{4k}(BU))$.
\end{enumerate}
%
\end{Theorem}

\begin{Remark} \label{rem:Yang+}
The simplification achieved in moving from Theorem \ref{thm:Yang} to
Theorem \ref{thm:Yang+} is perhaps surprising and rests on Theorem
\ref{thm:bernoulli}, which is a non{-}trivial fact about the differences of
reciprocals of Bernoulli numbers. Theorem~\ref{thm:Yang+} can be interpreted as a
statement about the characteristic numbers (signature and $p_{k}^2$)
of closed $(4k{-}1)$-connected almost complex smooth manifolds and the
bordism groups $\Omega_{8k}^{U\an{4k{-}1}}$.  It would be interesting to
see if there are further connections between number theory and the
characteristic numbers of closed $j$-connected almost complex
$n$-manifolds for other values of $j$ and $n$.
\end{Remark}

\begin{proof}[Proof of Theorem \ref{thm:Yang+}]
For $k=1, 2$, $B_1 = \frac{1}{6}$ and $B_2 = B_4 = \frac{1}{30}$, see
e.g.~\cite[p.\,12]{Hirzebruch66}.  Hence
\[ \frac{B_{2}+B_{1}}{B_{2}B_{1}} 
\cdot \frac{\sigma_Y}{2^{2}} = 9 \sigma_Y \quad \text{and} \quad
\frac{B_{4}-B_{2}}{B_{4}B_{2}}= 0\]
and Theorem \ref{thm:Yang} implies Theorem \ref{thm:Yang+} in these two cases.

The case $k > 2$ will follow from a result of Wall, a fact about Bernoulli numbers (see
Theorem~\ref{thm:bernoulli} in the Appendix) and from the evenness of
$\sigma_Y$.  As in Remark \ref{rem:Yang_and_us}, we regard $\tau_{Y*}$ as 
a cohomology class $\tau_{Y*} \in H^{4k}(Y)$ and define the integer
\[ \tau_Y^2 : = \an{(\tau_{Y*})^2, [Y]}. \]
%
%
Setting $a_k := \frac{3-(-1)^k}{2}$, Wall \cite[$(15)_m$]{Wall62} proved
that the $\wh A$-genus of a $(4k{-}1)$-connected $8k$-manifold $Y$ is given by the following formula:
\begin{equation} \label{eq:A2m}
\wh A_{2k}(Y) = 
\frac{a_k^2 \cdot 2^{4k-4}\cdot B_k^2 \cdot (2^{2k}-1)^2 \cdot \tau_Y^2 - k^2 \sigma_Y}{2^{4k{+}1} \cdot k^2 \cdot(2^{4k{-}1}-1)},
\end{equation}
where $\tau_Y^2 = \chi^2$ in Wall's notation. Now let $B_k = \frac{N_k}{D_k}$ where $N_k$ and $D_k$ denote
respectively the numerator and denominator of $B_k$ expressed in
lowest terms.  We recall from \cite[p.\,284]{Milnor&Stasheff74} that
$N_k$ is odd and that $D_k = 2 D'_k$ where $D'_k$ is odd: in fact
$D_k'$ is the product of odd primes $p$ such that $(p-1)$ divides
$2k$. Writing  $k = 2^j \cdot c$ for $c$ an odd integer,
$j \geq 0$, we rewrite \eqref{eq:A2m} as
\begin{equation} \label{eq:A2m-1}
 \wh A_{2k}(Y) = 
\frac{a_k^2 \cdot 2^{4k-6-2j} \cdot N_k^2 \cdot (2^{2k}-1)^2 \cdot \tau_Y^2 - (D'_k)^2 \cdot c^2 \cdot \sigma_Y}
{2^{4k{+}1} \cdot c^2 \cdot (D'_k)^2 \cdot (2^{4k{-}1}-1)  }. 
\end{equation}
Since $k > 2$ and $Y$ is $(4k{-}1)$-connected, the intersection form of $Y$ is 
even by \cite{Wall62}, and hence $\tau_Y^2$ is an even in integer.  In addition, if
$k$ is even, then for $Y$ to admit a stably complex structure, $\tau_Y$ 
must lie in $2H^{4k}(Y)$, and so $8$ divides $\tau_Y^2$.
Since $\wh A_{2k}(Y)$ is an integer, \eqref{eq:A2m-1} entails that $
2^{4k-3-2j}$ divides $\sigma_Y$.

To apply Theorem~\ref{thm:Yang} when $k$ is odd,
we must show that
\[ \Num \left( \frac{B_{2k}+B_{k}}{B_{2k}B_{k}} \cdot \frac{\sigma_Y}{2^{4k{-}2}}
 \right) \]
is even, where $\Num \left( \frac{a}{b}\right) $ denotes the
numerator of $\frac{a}{b}$, expressed in lowest terms. Since $k$ is odd, $j = 0$, and hence $2^{4k-3}$ divides $\sigma_Y$. 
Furthermore, the largest power of $2$ which divides $\Denom \left(
\frac{\sigma_Y}{2^{4k{-}2}} \right)$ is $2$.  But
\[ \Num \left( \frac{B_{2k}+B_k}{B_{2k}B_k} \right) = 
D_kN_{2k}+D_{2k}N_k = 2(D'_kN_{2k} + D'_{2k}N_k)\]
is divisible by $2^2$, and condition (1) in Theorem~\ref{thm:Yang} holds.

To apply Theorem~\ref{thm:Yang} when 
$k$ is even, we must show that
\[ \Num \left( \frac{B_{2k}-B_{k}}{B_{2k}B_{k}} 
\cdot \frac{4k \cdot \sigma_Y}{2^{4k}}\right) \]
is even. Since $k =
2^j c$ and $2^{4k-3-2j}$ divides $\sigma_Y$, the largest power of $2$
which can divide $\Denom \left( \frac{4k \cdot \sigma_Y}{2^{4k}}
\right)$ is $2^{j+1}$.  By Theorem~\ref{thm:bernoulli},
$2^{j+3}$ divides $\Num \left( \frac{B_{2k}-B_k}{B_{2k}B_k} \right)$,
which ensures that condition (2) in Theorem~\ref{thm:Yang} holds.
\end{proof}

To prove Theorem \ref{thm:S8k{-}1}, we shall need the following result.
\begin{Lemma} \label{lem:delJ_is_topological}
Let $q \equiv 3$~mod~$4$ and let $(W, J)$ be an almost complex $(2q{+}2)$-manifold with $\del W = S^{2q{+}1}$.  
The stable complex structure induced on the boundary, $(S^{2q{+}1}, S\del \acxs)$, is independent
on the choice of the almost complex structure $\acxs$ up to homotopy and depends only on the 
oriented diffeomorphism
type of $W$.
\end{Lemma}
\begin{proof}
Let $(W,J_0)$ and $(W,J_1)$ be two almost complex structures on the same topological Stein filling of $S^{2q{+}1}$, which is then a $q$-connected manifold. By the Hurewicz Theorem and repeated application of the Whitney trick we can find a basis of $H_{q{+}1}(W)$ consisting of primitive elements
$\{ x_1, \dots, x_n \}$ represented by embeddedings $f_i \colon S^{q{+}1} \hookrightarrow W$.  
The embedded spheres $f_i(S^{q{+}1})$
will intersect in the pattern determine by the intersection form of $W$,
which is unimodular, since the boundary of $M$ is a sphere, and is denoted by
\[ \lambda_W \colon H_{q{+}1}(W) \times H_{q{+}1}(W) \to \Z.\]

We now consider the boundary connected sum 
$\big( W\natural(-W), J_0\natural(-J_1) \bigr)$, 
given by reversing the orientation on $(W, J_1)$ and then 
attaching an almost complex $1$-handle.  This manifold is again
$q$-connected. By tubing together the two copies of the embeddings $f_i$,
taking care to reverse the orientation along the tube,
we obtain embeddedings $f_i \# (-f_i) \colon S^{q{+}1} \hookrightarrow W \natural (-W)$,
which represent a basis of the anti-diagonal summand
%
%
\[ L := \left\langle (x_1, -x_1) \, \dots\,, (x_n, -x_n) \right\rangle
\subset H_{q{+}1}(W\natural(-W)) = H_{q+1}(W) \oplus H_{q+1}(W). \]
We claim that this basis for $L$ is represented by disjoint embedded $(q{+}1)$-spheres
with trivial normal bundle.  The normal bundle of each embedding $f_i \# (-f_i)$,
$i = 1, \dots, n$, is isomorphic to the Whitney
sum of the normal bundle of $f_i$ and its inverse, and so is trivial.
Moreover, the intersection form of $W \natural (-W)$ is the orthogonal sum 
$\lambda_W \oplus -\lambda_W$, and so for all pairs $(i, j)$, 
the algebraic intersection of $(x_i, -x_i)$ with $(x_j, -x_j)$ is given by
\[ \lambda_{W \natural (-W)}((x_i, -x_i), (x_j, -x_j))
= \lambda_W(x_i, x_j) - \lambda_W(-x_i, -x_j) = 0 .\]
By further applications of the Whitney trick, we arrive at the required
disjoint embeddings 
$\phi_i \colon D^{q{+}1} \times S^{q{+}1} \hookrightarrow W \natural (-W)$ 
representing the given basis of $L$.

The stable complex structure induced on 
$\phi_i(D^{q{+}1} \times S^{q{+}1})$ by 
$\scxs = S(J_0 \natural (-J_1))$ may be regarded as an element 
$\scxs_i \in \pi_{q{+}1}(SO/U)$ and since $q \equiv 3$~mod~$4$, each $\scxs_i$ lies
in the image of the map $\pi_{q{+}1}(SO) \to \pi_{q{+}1}(SO/U)$.  Moreover, the condition that $q
\equiv 3$~mod~$4$ implies that the stabilization homomorphism
$\pi_{q{+}1}(SO(q{+}1)) \to \pi_{q{+}1}(SO)$ is onto by \cite{Kervaire60},
and hence we may reframe our embeddings to obtain new embeddings $\bar
\phi_i$ so that each $\scxs_i$ is trivial.  It follows that there is no
obstruction to extending the stable complex structure induced by 
$J_0 \natural (-J_1)$ on $W \natural (-W)$ over a handle attachment along $\bar
\phi_i$.  That is, we may perform stably complex surgeries on the
embeddings $\bar \phi_i$: see \cite[Section 2.3]{BCS1}.
The trace of these surgeries is a stably complex bordism, relative to
the boundary, to a simply connected homology ball, which is in turn a
topological ball: see \cite[Lemma 7.1]{Kervaire-Milnor63}.
Moreover, the stable almost complex structure on the
boundary is equal to the stabilization of $\del \acxs_0 \# (-\del
\acxs_1)$. It follows that $S(\del \acxs_0 \#(-\del \acxs_1)) $ is the
standard stable complex structure and thus that $S\del \acxs_0 = S\del
\acxs_1$.
\end{proof}	

\begin{proof}[Proof of Theorem \ref{thm:S8k{-}1}]
Suppose that $k  \geq2$.
Let $(W, J)$ be a Stein filling with boundary $S^{8k{-}1}$, and
consider  the smooth closed
oriented manifold $X$
obtained by adding the $8k$-disc to $W$ via the identity map:
\[ X : = W \cup_{\rm Id} D^{8k}.  \]
We consider the manifold $X = W \cup_{\Id}
D^{8k}$. Note that since $W$ admits an almost complex structure
$\acxs$ by hypothesis, we have $\im(\tau_{X*}) = \im(\tau_{W*})
\subset F_*(\pi_{4k}(BU))$.  It follows from Theorem \ref{thm:Yang+} that $X$ also admits a stable complex
structure $\scxs_X$.

Now take the resulting stably complex manifold $(X, \scxs_X)$ and
remove a small open disc.  The outcome is a smooth oriented manifold
diffeomorphic to $W$ with an induced stable complex structure
$\scxs_W$.  Since $\scxs_W$ extends to $X$, we conclude that the
induced stable complex structure $\del \scxs_W$ on $S^{8k{-}1}$ is
homotopic to $\scxs_0$.  Now by Lemma~\ref{lem:delJ_is_topological}, 
the stable complex structures $S \del J$ and $\del \scxs_W$ are
homotopic and hence $S \del J$ is homotopic to $\scxs_0$.  This shows
that only the standard stable complex structure on $S^{8k{-}1}$ admits a
Stein filling, which proves Theorem \ref{thm:S8k{-}1}.
\end{proof}


\subsection{A description of $(S^{8k{-}1}, \acs_{ex})$ } \label{subsec:explicit}
In this subsection we give an explicit description of an almost contact structure
$\acs_{ex}$ on $S^{8k{-}1}$ which stabilizes to $(S^{8k{-}1}, \scxs_{ex})$ when $k \geq 2$. 
Recall that $\pi_{8k{-}1}(SO) \to \pi_{8k{-}1}(SO/U)$ is onto,
and that by \cite{Kervaire60}, for $k \geq 2$,
the stabilization homomorphism
\[ \pi_{8k{-}1}(SO(8k{-}2)) \to \pi_{8k{-}1}(SO) \]
is also onto.  Let $f \colon (D^{8k{-}1}, S^{8k{-}2}) \to (SO(8k{-}2),
\Id)$ be a smooth map representing a class $[f] \in
\pi_{8k{-}1}(SO(8k{-}2))$ where $[f]$ stabilizes to a generator of
$\pi_{8k{-}1}(SO)$.  Let $\xi_{std} \subset TS^{8k{-}1}$ be the
oriented hyperplane distribution given by the standard contact
structure on $S^{8k{-}1}$ and let $J_{std}$ be the complex structure
on $\xi_{std}$ induced by the choice of a contact form.  We observe
that we can use $f$ to define a vector bundle automorphism
\[ \alpha_f \colon \xi_{std} \cong \xi_{std} \]
where $\alpha_f$ is the identity on all fibres outside a small 
$(8k{-}1)$-disc $D \subset S^{8k{-}1}$ and on $TS^{8k{-}1}|_D \cong D \times \R^{8k{-}1}$
we use $f$ to twist $\xi_{std}$ in the obvious way.  We can then use $\alpha_f$ to pull-back
the complex structure $J_{std}$ on $\xi_{std}$ and obtain $\alpha_f^*(J_{std})$.  Clearly
$(\xi_{std}, J_{std})$ and $(\xi_{std}, \alpha_f^*(J_{std}))$ are isomorphic complex vector bundles but
since $\alpha_f$ is not homotopic to a unitary automorphism of $(\xi_{std}, J_{std})$, 
it follows
that $(\xi_{std}, J_{std})$ and $(\xi_{std}, \alpha_f^*(J_{std}))$ are not homotopic as complex structures on $\xi_{std}$. 
Indeed, even after stabilization $\alpha_f$ is not homotopic to a unitary automorphism
and so the almost contact structure 
\begin{equation} \label{eq:ex}
\varphi_{ex} := (\xi_{std} \subset TS^{8k{-}1}, \alpha_f^*(J_{std})) 
\end{equation}
stabilizes to the stable complex structure on $S^{8k{-}1}$ given by acting on $\scxs_{std}$ with the generator of $\pi_{8k{-}1}(SO/U) \cong \Z/2$.  Hence we have proven

\begin{Lemma} \label{lem:zeta_ex}
For $k \geq 2$, the almost contact structure $(S^{8k{-}1}, \varphi_{ex})$ of \eqref{eq:ex}
stabilizes to the stable complex structure $(S^{8k{-}1}, \zeta_{ex})$. \qed
\end{Lemma}

The examples $(S^{8k{-}1}, \acs_{ex})$ above and also the examples $(STS^{4k{-}1}, \acs)$ 
with $c_k(\varphi) \neq 0$ from Lemma \ref{lem:cotangent} are interesting examples of $(q{-}1)$-connected
$(2q{+}1)$-dimensional almost contact manifolds which are not Stein fillable.
The Stein fillability of such manifolds was studied in \cite{Geiges93, Geiges97}.
In \cite{BCS3} we take up this question in the context of Theorem \ref{thm:stein} 
by systematically studying the bordism groups $\Omega_{2q{+}1}(B^{q{-}1}_\scxs; \eta^{q{-}1}_\scxs)$.

\appendix
\section*{Appendix: $2$-adic valuation of differences of the Bernoulli numbers:\\ 
By Bernd C. Kellner}
\label{sec:bernoulli}
\setcounter{section}{1}
\setcounter{Theorem}{0}

Let $B_k$ be the $k^{th}$ Bernoulli number with topologist's indexing and sign
conventions as in \cite[p.\,12]{Hirzebruch66} and
\cite[Appendix B]{Milnor&Stasheff74}. In particular, we have
\[
  \begin{aligned}
      B_1 & = \frac{1}{6},  & B_2 & = \frac{1}{30},     & B_3 & = \frac{1}{42},
    & B_4 & = \frac{1}{30}, \\
      B_5 & = \frac{5}{66}, & B_6 & = \frac{691}{2730}, & B_7 & = \frac{7}{6},
    & B_8 & = \frac{3617}{510}.
  \end{aligned}
\]
Given a fraction $\frac{a}{b}$, let $\Num \left( \frac{a}{b} \right)$ and
$\Denom \left( \frac{a}{b} \right)$ denote respectively the numerator and the
denominator of $\frac{a}{b}$, when expressed in lowest terms. In this Appendix
we prove the following theorem about Bernoulli numbers, which is the essential
number-theoretic input to the proof of Theorem~\ref{thm:Yang+}.

\begin{Theorem} \label{thm:bernoulli}
Suppose that $k$ is even and write $k = 2^j c$,
where $c$ is odd and $j \geq 1$. Then
\[
  2^{j+3} \mid \Num \left( \frac{B_{2k} - B_k}{B_{2k}B_k} \right).
\]
\end{Theorem}


Let $p$ be any prime and let $\Z_p$ denote the ring of $p$-adic integers. As usual, define the $p$-adic valuation of $s \in \Z_p$ by $\ord_p s$,
such that $s = u \, p^{\ord_p s}$ where $u \in \Z_p^\times$ is a unit.

We will prove Theorem~\ref{thm:bernoulli} later, since we first need to show
some $p$-adic properties of the Bernoulli numbers. From now on, it is more
convenient to switch to the notation of signed and even{-}indexed Bernoulli
numbers $\BN_n$ as commonly used in number theory. They may be defined by
the generating function
\[
  \frac{t}{e^t-1} = \sum_{n \geq 0} \BN_n \frac{t^n}{n!}, \quad |t| < 2 \pi.
\]

These numbers are rational and $\BN_n = 0$ for odd $n > 1$.
The even{-}indexed Bernoulli numbers alternate in sign, such that
$(-1)^{\frac{n}{2}+1} \BN_n > 0$ for even $n > 0$. Accordingly
\[
  (-1)^{n{+}1} B_n = \BN_{2n}, \quad n \geq 1.
\]

The famous theorem of von Staudt and Clausen
\cite[Theorem~3, p.~233]{Ireland&Rosen90} asserts for even $n \geq 2$, that
\begin{equation} \label{eq:bernoulli_denom}
  \BN_n + \sum_{p-1 \mid n} \frac{1}{p} \in \Z, \quad \text{which implies that}
  \quad \Denom(\BN_n) = \prod_{p-1 \mid n} p.
\end{equation}

Let $n \geq 2$ be even. If $p-1 \mid n$, then we obtain by
\eqref{eq:bernoulli_denom} that
\[
  \BN_n + \frac{1}{p} \in \Z_p,
\]
whereas we already have $\BN_n \in \Z_p$ in the case $p-1 \nmid n$.
Both cases imply that
\[
  \BN_n - \BN_m \in \Z_p,
\]
whenever $n, m \geq 2$ are both even and satisfy $n \equiv m \pmod{p-1}$.
As an easy consequence, iterated finite differences of a sequence of
Bernoulli numbers $\BN_n$ are $p$-integers, assuming that all indices
are even and congruent $\bmod\ p-1$.
Now, we consider the special case $p=2$, where we use the following
more general result of Carlitz.

\begin{Theorem}[{Carlitz \cite[Theorem~7]{Carlitz60}}] \label{thm:carlitz}
If $n \geq 2$ is even, $r \geq 1$, and $2^{e-1} \mid w$ with $e \geq 2$, then
\[
  \sum_{s=0}^r (-1)^s \binom{r}{s} 2 \BN_{n{+}sw} \equiv 0
    \pmod{ \gcd(2^{n{-}1},2^{re+\lambda}) },
\]
where $\lambda = \min (r-1, r-r'+3)$ and $2^{r'} \leq 2r < 2^{r'+1}$. 
%
\end{Theorem}

Note that the sum above describes an iterated finite difference with
increment $w$. As mentioned above, this sum still lies in $\Z_2$, if we cancel
the factor 2 that occurs. We can rewrite this result as follows.

\begin{Corollary} \label{cor:carlitz}
If $n, w \geq 2$ are both even and $r \geq 1$, then
\[
  \ord_2 \left( \sum_{s=0}^r (-1)^s \binom{r}{s} \BN_{n{+}sw} \right)
    \geq \min( n{-}2, re + \lambda - 1 ),
\]
where $e = 1 + \ord_2 w \geq 2$, $l = \lfloor \log_2 r \rfloor \geq 0$,
and $\lambda = \min( r-1, r-l+2 ) \geq 0$. 
%
\end{Corollary}

\begin{Proposition} \label{prop:bernoulli_ord2}
If $m>n \geq 2$ are both even, then
\[
  \ord_2 \left( \frac{1}{\BN_n} - \frac{1}{\BN_m} \right)
    = 2 + \ord_2 ( \BN_n - \BN_m )
    \geq \min ( n, 2 + \ord_2( m - n ) ).
\]
\end{Proposition}

\begin{proof}
We first observe that
\[
  \ord_2 \left( \frac{1}{\BN_n} - \frac{1}{\BN_m} \right)
    = \ord_2 \left( \frac{\BN_n - \BN_m}{\BN_n \BN_m} \right)
    = 2 + \ord_2 ( \BN_n - \BN_m ),
\]
since by \eqref{eq:bernoulli_denom} we have $\ord_2( \Denom(\BN_n \BN_m) ) = 2$.
Using Corollary~\ref{cor:carlitz} with parameters $r = 1$ and $w = m - n$,
we then infer that
\[
  \ord_2 ( \BN_n - \BN_m ) \geq \min ( n{-}2, \ord_2( m - n ) ),
\]
completing the proof.
\end{proof}

\begin{proof}[Proof of Theorem~\ref{thm:bernoulli}]
Recall that $k = 2^j c$ where $c$ is odd and $j \geq 1$.
Since $k$ is even, the Bernoulli numbers $\BN_{2k}$ and $\BN_{4k}$
have the same sign.
Thus, we can apply Proposition~\ref{prop:bernoulli_ord2} to obtain that
\begin{align*}
  \ord_2 \left( \frac{B_{2k} - B_k}{B_{2k}B_k} \right)
    = 2 + \ord_2 ( \BN_{2k} - \BN_{4k} )
    \geq \min ( 2k, 3 + \ord_2 k ) = 3 + \ord_2 k.
\end{align*}
The last step follows by a simple counting argument.
Since $\ord_2 k = j$, this gives the result.
\end{proof}

In the Summer of 2013, Theorem~\ref{thm:bernoulli} arose as a conjecture.
At the same time, it was independently proved by Karl Dilcher and
the author of this appendix using results of Carlitz.


\end{document}